\documentclass[a4paper]{amsart}

\usepackage{amsmath,amssymb,url}
\usepackage[hidelinks]{hyperref}

\theoremstyle{definition}    
\newtheorem*{Example}{Example} 
\newtheorem*{Rem}{Remark}

\newcommand{\Lin}{\mathop{\mathcal Lin}\nolimits}

\newcommand{\GCD}{\operatorname{GCD}}
\newcommand{\clos}{\operatorname{clos}}
\newcommand{\Span}{\operatorname{span}}
\newcommand{\Ker}{\operatorname{Ker}}
\newcommand{\loc}{\mathrm{loc}}
\newcommand{\odd}{\mathrm{odd}}
\newcommand{\supp}{\mathop{\mathrm{supp}}}
\newcommand{\Mult}{\mathit{Mult}}
\newcommand{\one}{\mathbf 1}
\newcommand{\No}{N$^\circ$}
\newcommand{\BB}{\mathbb{B}}
\newcommand{\CC}{\mathbb{C}}
\newcommand{\DD}{\mathbb{D}}
\newcommand{\NN}{\mathbb{N}}
\newcommand{\QQ}{\mathbb{Q}}
\newcommand{\RR}{\mathbb{R}}
\newcommand{\TT}{\mathbb{T}}
\newcommand{\ZZ}{\mathbb{Z}}
\newcommand{\fD}{\mathfrak{D}}

\newskip\thmsep
\newtheoremstyle{mystyle}
    {\topsep}
    {\topsep}
    {\it}%   Body font
    {\parindent}%   Indent amount (empty = no indent, \parindent = para indent)
    {\bfseries}% Thm head font
    { }%  Punctuation after thm head
    { }%  Space after thm head (\newline = linebreak)
    {\thmnumber{#2. }\thmname{#1}\thmnote{ #3}.}%   Thm head spec
\theoremstyle{mystyle}

\newtheorem{Lemma}[subsection]{Lemma}
\newtheorem{Proposition}[subsection]{Proposition}
\newtheorem*{Claim}{Claim}

\newtoks\thehProclaim
\newtheorem{Proclaim}[subsection]{\the\thehProclaim}
\newenvironment{proclaim}[1]{\thehProclaim{#1}\begin{Proclaim}}{\end{Proclaim}}

\newtheoremstyle{myremstyle}
    {\topsep}
    {\topsep}
    {\rm}%   Body font
    {\parindent}%   Indent amount (empty = no indent, \parindent = para indent)
    {\bfseries}% Thm head font
    { }%  Punctuation after thm head
    { }%  Space after thm head (\newline = linebreak)
    {\thmnumber{#2. }\thmname{#1}\thmnote{ #3}.}%   Thm head spec
\theoremstyle{myremstyle}
\newtheorem{Remark}[subsection]{Remark}
\newtheorem{Comments}[subsection]{Comments}

\newtoks\thehSubS
\newtheorem{SubS}[subsection]{\the\thehSubS}
\newenvironment{Subsection}[1]{\thehSubS{#1}\begin{SubS}}{\end{SubS}}

\let\bull=\square

\begin{document}

\title[On the Cyclicity of Dilated Systems in Lattices]%
{On the Cyclicity of Dilated Systems in Lattices:
Multiplicative Sequences, Polynomials, Dirichlet-type Spaces 
and Algebras}

\author{Nikolai Nikolski} 

\address{Universit\'e de Bordeaux, Institut de Math\'ematiques}

\email{nikolski@math.u-bordeaux.fr}

\dedicatory{To the memory of Mohamed Zarrabi, an unforgettable colleague 
and remarkable man}

\begin{abstract}
The aim of these notes is to discuss the completeness of the dilated systems in a most general framework of an arbitrary sequence lattice $X$, including weighted $\ell^p$ spaces. In particular, general multiplicative and completely multiplicative sequences are treated. After the Fourier--Bohr transformation, we deal with the cyclicity property in function spaces on the corresponding infinite dimensional Reinhardt domain $\DD^\infty_{X'}$. Functions with (weakly) dominating free term and (in particular) linearly factorable functions are considered. The most attention is paid to the cases of the polydiscs $\DD^\infty_{X'}|\CC^N=\DD^N$ and the $\ell^p$-unit balls $\DD^\infty_{X'}|\CC^N=\BB_p^N$, in particular to Dirichlet-type and Dirichlet--Drury--Arveson-type spaces and algebras, as $X=\ell^p(\ZZ_+^N,(1+\alpha)^s)$, $s=(s_1,s_2,\dots)$ and $X=\ell^p(\ZZ_+^N,(\frac{\alpha!}{|\alpha|!})^t(1+|\alpha|)^s)$, $s,t\geq 0$, as well as to their infinite variables analogues. We privileged the largest possible scale of spaces and the most elementary instruments used. 
\end{abstract} 

\maketitle

\tableofcontents

% Contents: 
% 1. Introduction. 10th proof of a Haar Lemma. 
% 2. Invariant sequence lattices (ISL). 
% 3. Totally multiplicative sequences are always cyclic. 
% 4. The Fourier-Bohr transform in complex domain. 
% 5. Cyclicity of (relatively) multiplicative sequences. 
% 6. Beurling's necessary condition, and a few cases where it is sufficient. 
% 7. Factorable functions in Dirichlet-type spaces and algebras in polydiscs. 
% 8. Factorable functions in Dirichlet-Drury-Arveson spaces and algebras in $ \ell^{p}$-unit balls. 
% 9. Brief notes on the history of the periodic Dilation Completeness Problem (DCP). 
% References 
 
\section{Introduction. 10th proof of a Haar Lemma}
 
Given a Banach space of complex sequences $ X$, the \emph{dilation 
completeness problem} (DCP) consists in a description of \emph{cyclic 
elements $x=(x_{k})_{k\geq 1}\in X$ for the dilation semigroup 
$(D_{n})_{n\geq 1}$}, i.e. $x\in X$ such that $E_{x}= X$ 
where 
$$
E_{x}=: \Span_\sigma \big(D_{n}x :  n=1,2,3,\dots \big) .
$$
 
Below, we often suppose that \emph{the space $X$ is an {\bf ``ISL''} --- 
(Dilation) Invariant Sequence Lattice}, i.e. a Banach space of complex 
sequences such that:
 
{\it (i)} 
with every $x\in X$ it contains all sequences $y=(y_n)_{n\geq 1}$ having 
$|y_n|\leq |x_n|$, $\forall n\in\NN$ and $\|y\|_X\leq \|x\|_X$, 
 
{\it (ii) the dilations $D_n$}
$$
D_n\left(\sum_k x_k\delta_k\right) := \sum_k x_k\delta_{nk}
$$
are well defined on $X$, $x\in X \Rightarrow D_n x\in X, \forall n\geq 1$, 
and
 
{\it (iii) $\delta_1 := (1,0,0,\dots)\in X$} and so all \emph{standard unit 
vectors $\delta_k$ (and their linear hull $X_{00}$),}
$$
\delta_k = (\delta_{k,j})_{j\geq 1} \quad (\delta_{k,j} 
\text{ being the Kronecker delta}),
$$
are contained in $X$. 
 
We distinguish \emph{``normalized''} {\bf ISL} (or {\bf N-ISL}) by the 
property $\|\delta_k\|_X = 1$, $\forall k$. In this case, 
$$ 
\ell^1(\NN)\subset X\subset \ell^\infty(\NN),
$$ 
with the corresponding inequalities for the norms. 
 
The notation $\Span_\sigma (\dots)$ means \emph{the closed linear span of 
a set} of vectors $(\dots)$ \emph{in $X$}, closed for the \emph{weak 
topology $\sigma(X,X')$} induced by the \emph{dual lattice $X'$} with 
respect to the \emph{standard duality}
$$ 
\langle x,y\rangle = \sum_{n\geq 1}x_n y_n ,
$$
where $X'$ stands for 
$$
X' := \left\{ y= (y_{n}) : \sum_{n\geq 1} |x_ny_n| < \infty, \forall x\in X \right\} .
$$ 
Notice that always $X'\subset X^*$ (the Banach dual space), and if the set 
$X_{00}$ of all finitely supported sequences (the linear span of 
$\delta_k$, $k=1,2,\dots$) is dense in $X$, then $X'= X^*$. Clearly, if 
$X'=X^*$ then $\Span_\sigma = \Span_X$, the norm closed linear hull in 
$X$. In any case, the dual norm on $X'$ is defined by 
$\|y\|_{X^*}=\sup\{|\langle x,y\rangle| : \|x\|_X\leq 1\}$. 
 
As examples of {\bf N-ISL}, we can mention the \emph{rearrangement 
invariant sequence spaces}, {\bf RIS} (see details below, as well as 
\cite{LiT1979}), and in particular, $\ell^p(\NN)$ spaces 
($1\leq p\leq\infty$), Orlicz, Lorentz, Marcinkiewicz sequence spaces $X$. 
As {\bf ISL}'s, one can consider some weighted lattices $X_w$ ($x\in X_w$ 
iff $(x_n w_n)\in X$), see below. Quite recently, there were discovered 
many interesting facts on the DCP related to Dirichlet series and the 
analytic number theory, mostly on the Hilbert space $X=\ell^2(\NN)$, see 
for example \cite{HLS1997,Que2015,Nik2012,Nik2019,DGu2020}, as well as the 
references in Sections 4.4, 5.4, 6.8, 6.10 and 9 below.
 
{\bf Acknowledgements.} 
The author is mostly grateful to several colleagues who kindly regarded 
preliminary versions of the manuscript and made very valuable remarks: 
Evgueny Abakumov (Marne-la-Vall\'ee), Alexander Borichev (Aix-Marseille), 
Herv\'e Queff\'elec (Lille), and Rachid Zarouf (Aix-Marseille). In 
particular, A.~Borichev suggested a ``generalized Vandermond's identity'' 
for the proof of Lemma 8.2, and H.~Queff\'elec helped with an attribution 
of A.~Haar's Lemma. 
 
\centerline{\bf *{*}*}
 
Later on, we return to the general setting described above, but just now, 
consider the DCP in the space $X=\ell^\infty(\NN)$ of all bounded complex 
sequences, with the usual $\sup$ norm. The following statement is our 
departure point for this paper.
 
\begin{proclaim}{Haar's Lemma}
If $a_k\in\CC$, $\sum_{k\geq 1} |a_k|<\infty$ 
and $\sum_{k\geq 1}a_{nk}=0$ ($\forall n\geq 1$), 
then $a_j=0$, $\forall j\geq 1$.
\end{proclaim}
 
It seems that the lemma first appeared in G.~Polya and G.~Szeg\H{o} problem 
book \cite{PSz1925}, Part I, \No 129 with a (rather involved) proof 
attributed to Alfred Haar (without detailed references). Much later, in 
\cite{Que2015}, p.~294, the proposition is referred to as just ``a simple 
classical lemma'' and is proved making use some specific procedures 
(related to the sieve of Eratosthenes and the classical M\"obius inversion 
formula --- see below). In \cite{Que2015}, the lemma is then used for the 
proof of a A.~Wintner result quoted in Remark 3.2 below.

Clearly, by the duality, Lemma 1.1 means that \emph{the constant sequence
$$ 
\one = (1,1,\dots)
$$
is a $(D_n)$-cyclic vector in $\ell^\infty(\NN)$ for the 
weak-* topology $\sigma(\ell^\infty,\ell^1)= \sigma(\ell^\infty,(\ell^\infty)')$}.

In fact, historically, similar cyclicity properties in the space $\ell^2$ 
for more general \emph{totally (or, completely) multiplicative sequences 
$x$}, 
$$
x_{nk} = x_n x_ k \quad (\forall n,k \in \NN)
$$
were (re)discovered many times, under various hypotheses and forms, and 
usually with different proofs. I found 9 such proofs in the literature, 
starting from the initial Haar's lemma, and then in A.~Wintner 
\cite{Win1944}, N.~P.~Romanov \cite{Rom1946}, N.~I.~Ahiezer \cite{Ahi1947}, 
P.~Hartman \cite{Har1947} (the most complete treatment), V.~Ya.~Kozlov 
\cite{Koz1948b}, H.~Hedenmalm--P.~Lindquist--K.~Seip \cite{HLS1997}, 
N.~Nikolski \cite{Nik2012} (with \cite{Nik2018}), H.~Queff\'elec 
\cite{Que2015}. Below, a tenth proof of that fact is given in much more 
general framework, surely the shortest one and completely elementary. 
In subsequent sections, we similarly prove some more general statements. 

Note that a lattice $X$ always contains the set $X_{00}$ of finitely 
supported sequences $x = \sum_k x_k\delta_k$, i.e. finite linear 
combinations of the standard unit vectors $\delta_k$. The set $X_{00}$ is 
dense in $X$ for $\sigma (X,X')$ topology.
 
\begin{proclaim}{Extended Haar's Lemma}
Let $x = (x_k)_{k\geq 1}$, $x_1=1$, be a bounded (complex) totally 
multiplicative sequence. If $\sum_{k\geq 1}|a_k| < \infty$ and 
$$
\sum_{k\geq 1} x_k a_{nk}= 0 \quad (\forall n\geq 1), \text{ then } 
a_j=0, \ \forall j\geq 1.
$$
\end{proclaim}

\begin{proof} 
Let $p_1=2$, $p_2=3$, \dots, $p_s$, \dots{} be all integer primes. Setting 
$x(0)=x$, define $x(s)$, $s=1,2,\dots$, recurrently as 
$$ 
x(s)= (id-x_{p_s}D_{p_s})x(s-1), \quad s=1,\dots
$$ 
Obviously, $x(s)\in E_x$ for every $s= 0,1,\dots$ By multiplicativity, all 
coordinates of $x(1)$ with indices in $2\NN$ are zero, whereas all 
others are the same as for $x(0)$, and next, all coordinates of $x(2)$ 
with indices in $3\NN$ are zero, whereas all others are the same as 
for $x(1)$, etc. 

Since (clearly) $\|x(s)\|_\infty\leq \|x(s-1)\|_\infty$ and all 
coordinates of $x(s)$ with indices between $2$ and $p_s$ are zero, we 
obtain $(\sigma^*)\lim_{s\to\infty}x(s)=\delta_1$, where 
$\delta_1=(1,0,0,\dots)$ and $\sigma^*=\sigma(\ell^\infty,\ell^1)$ stands for 
the $*$-weak topology. 

It follows that $\delta_1\in E_x$, and then 
$\delta_n=D(n)\delta_1\in E_x$, $\forall n$. Since $E_x\perp a$ (in the 
sense of the introduced duality), $a=0$.
\end{proof}

Below, we extend Lemma 1.2 to general {\bf ISL}s and to (partially) 
multiplicative sequences, and list some more observations on the DCP in 
the $D_n$-invariant lattices ({\bf ISL}).  
 
\section{Invariant sequence lattices (ISL)}
 
As it was mentioned, \emph{in this paper, the spaces} where the semigroup 
$(D_n)_{n\in \NN}$ acts, are the $D_n$-\emph{invariant sequence lattices} 
{\bf (ISL)} defined in points (i)--(iii) of Section 1 above. If in 
addition, the norm in $ X$ is permutation invariant, 
$$ 
\| (x_n) \| = \| (x_{\sigma(n)}) \| 
\quad (\text{for every bijection } \sigma: \NN\to\NN),
$$
we speak on a \emph{rearrangement invariant space $X$} {\bf (RIS)}. 
The classical examples of {\bf RIS} are $X=\ell^p=\ell^p(\NN)$, as well as 
the Orlicz and Lorentz sequence spaces $X$ (see \cite{LiT1979}). 

Clearly, a {\bf RIS} is a {\bf N-ISL}, a normalized {\bf ISL} as it is 
defined in Section 1. As further examples of {\bf ISL}, one can consider 
the corresponding weighted spaces: given a {\bf N-ISL} $X$, let 
$$ 
X_w = \{ (x_n) : (x_n w_n) \in X \},
$$
where $ (w_n)$ is a positive weight satisfying 
$$ 
\sup_{k\geq 1} \frac{w_{nk}}{w_k} < \infty, \quad \forall n\in\NN.
$$
Besides its own norm topology, we consider on $X$ a ``weakened topology'' 
$\sigma (X,X')$; it can be different from the standard weak topology 
$\sigma (X,X^{*})$, especially in the case where $X'$ is rather a predual 
space than the dual one $X^*$ (as for $X=\ell^\infty(\NN)$). Notice that 
the subspace $X_{00}$ is always $\sigma(X,X')$ dense in $X$ (indeed, if 
$y\in X'$ and $\langle\delta_n,y\rangle = 0$, $\forall n$ then $y=0$).  

We are using the following obvious coordinate-wise meaning of the 
$X'$-weak convergence. 
 
\begin{Lemma}
Let $X$ be an {\bf ISL}, and $\sigma = \sigma(X,X')$. 

(1) (Coordinate dominated convergence) If $x,x^{(n)}\in X$ and if 
$|x^{(n)}_j|\leq|x_j|$, $\forall n,j$ and $\lim_n x^{(n)}_j=0$, 
$\forall j$, then $(\sigma)\lim_n x^{(n)}=0$. 

(2) Assume that the finitely supported sequences $X'_{00}$ are norm dense 
in $X'$ (for example, $X_{00}$ is dense in $X$ and $X$ is reflexive), and 
let $(x^{(n)})_{n\geq 1}$ be a sequence in $X$. Then, 
$(\sigma)\lim_n x^{(n)}=x$ if and only if $\sup_n\| x^{(n)}\|_X < \infty$ 
and $\lim_n x_j^{(n)}=x_j$, $\forall j\geq 1$. 
\end{Lemma}
 
\begin{proof}
In fact, the proofs of the both claims are classical exercises, the first 
one with a use of the Lebesgue dominated convergence: 
$\lim_n \sum_jx^{(n)}_jy_j = \sum_jx_jy_j$.  
\end{proof}
 
Notice that without any hypothesis, the coordinatewise convergence and 
norm boundedness are not sufficient for the weak $\sigma(X,X')$ 
convergence (for instance, for $X=\ell^1$).
 
\section{Totally multiplicative sequences are always cyclic}
 
A sequence of complex numbers $x=(x_k)_{k\in\NN}$ is \emph{totally 
multiplicative} if $x_{nk}=x_nx_k$ for all $n,k\in\NN$ (so, $x_1=1$, if 
$x\neq 0$); $x$ is called also a semicharacter of $\NN$. 
 
\begin{Proposition}
Any nonzero semicharacter $x$, $x\neq 0$, lying in an {\bf ISL $X$}, is 
$\sigma(X,X')$-cyclic for the dilation semigroup $(D_n)$, i.e. $E_x=X$, 
where
$$ 
E_x = \Span_\sigma (D_n x : n\in\NN) .
$$
(So, in the case when $ X'= X^{*}$, it is norm cyclic). 

In particular (an extended form of the above Haar's Lemma), if such a 
sequence $x$ is in an {\bf ISL $X$} and $a\in X'$, and 
$$ 
\sum_{k\geq 1}a_{nk}x_k = 0 \quad (\forall n\geq 1), \text{ then } 
a_j=0, \ \forall j\geq 1.
$$
\end{Proposition}
 
\begin{proof}
We use the sequences $x(s)\in E_x$ considered in the proof of ``Extended 
Haar's Lemma'' (1.2 of Section 1), 
$$ 
x(s) =: (id-x_{p_s}D_{p_s})x(s-1),\quad s=1,\dots, x(0)= x.
$$ 
Then, $|x(s)_j|\leq |x_j|$ for every $s$ and $j$, so that --- by Lemma 2.1 
--- $(\sigma)\lim_s x(s)=\delta_1\in E_x$, and hence $\delta_n\in E_x$ 
($\forall n\geq 1$). But as it was mentioned, $\Lin(\delta_n:n\geq 1)$ is 
$\sigma(X,X')$-dense in $X$. 
\end{proof}
 
\begin{Remark}
Notice, that for $X=\ell^p$, $1\leq p<\infty$ and for every reflexive 
{\bf ISL}, Proposition 3.1 gives $D_n$-cyclicity also for the norm 
convergence. Moreover, if $x\in X$ is $\sigma(X,X')$-cyclic in a 
{\bf ISL $X$}, it is $ \sigma (Y,Y')$-cyclic in any larger {\bf ISL} 
$Y\supset X$. 

As an example to 3.1, a sequence $(k^{-\gamma })_{k\geq 1}$ is 
dilation-cyclic in any {\bf ISL $X$} where it is contained. In particular, 
forestalling a short discussion in Section 9 below, it follows that a 
function sequence $\varphi(nx)$, $n= 1,2,\dots$ is complete in $L^2(0,1)$ 
if $\varphi = \sum_{k\geq 1}\frac{\sin(k\pi x)}{k^\gamma}$, $\gamma>1/2$ 
(A.~Wintner \cite{Win1944}; later on, this fact was repeated in many other 
sources).
\end{Remark}
 
Now, we extend our language introducing the standard additive writing for 
the semigroup $\NN$, and recalling the \emph{Fourier--Bohr transform}. 
 
\section{The Fourier--Bohr transform in complex domain}
 
First, we recall the language of the dilation semigroup. Consider the set 
of positive integers $\NN$ as a subsemigroup of a (multiplicative) 
group of positive rationals $\QQ_+$ endowed with the discrete topology. 
The latter is canonically isomorphic to the (additively written) group 
$\ZZ(\infty)$ --- a subgroup of 
$\ZZ^\infty=\ZZ\times\ZZ\times\dots$ consisting of finitely 
supported sequences: $\alpha\in\ZZ(\infty)$ if and only if 
$\alpha = (\alpha_j)_{j\geq 1}$, $\alpha_j\in\ZZ$ and 
$\supp(\alpha) = \{j\in\NN : \alpha_j\neq 0\}$ is finite. The 
isomorphism is given by the Euclid prime decompositions: 
$$ 
r\longmapsto\alpha (r) \quad (r\in\QQ_+), \text{ where }
r=p^{\alpha(r)}:=\prod p_j^{\alpha_j},
$$ 
$p= (p_1,p_2,\dots)$ stands for the sequence of the consecutive primes 
($p_1=2$, $p_2=3$, etc.) and $\alpha(r)=(\alpha_1,\alpha_2,\dots)$, 
$\alpha_j\in\ZZ$. The $\alpha$-range of the subsemigroup 
$\NN\subset\QQ_+$ is 
$\alpha(\NN)=\ZZ_+(\infty)=:\ZZ(\infty)\cap\ZZ_+^\infty$. 

The (compact) dual group of $\ZZ(\infty)$ (and so of $\QQ_+$) is 
$\TT^\infty = \TT\times\TT\times\dots$ (with the infinite product 
topology), where the duality is written in a natural way: 
$$ 
\{\alpha,\zeta\} = \zeta^\alpha := \prod\zeta_j^{\alpha_j}, \quad 
\alpha = (\alpha_j)\in\ZZ(\infty) \text{ and } 
\zeta = (\zeta_j)_{j\geq 1}\in \TT^\infty,
$$ 
and respectively, 
$$ 
\{r,\zeta\} = \zeta^{\alpha (r)} := \prod\zeta_j^{\alpha_j}, 
\alpha(r) = (\alpha_j)\in\ZZ(\infty), 
r = p^{\alpha(r)} \text{ and } \zeta\in\TT^\infty.
$$ 
The corresponding Fourier (Laplace) transformation is called \emph{the 
Bohr transform} (or \emph{the Bohr lift}, \cite{Boh1913}). For a finitely  
supported function $h$ on $\QQ_+$ (or, respectively, $h$ on 
$\ZZ(\infty)$), it is defined as 
$$ 
Bh(\zeta) = \sum_{r\in\QQ_+} h(r)\zeta^{\alpha(r)}; \text{ resp., }
Bh(\zeta) = \sum_{\alpha\in\ZZ(\infty)} h(\alpha)\zeta^\alpha, 
\quad \zeta\in\TT^\infty.
$$ 
The shift operations for functions on $\QQ_+$ and $\ZZ(\infty)$ are 
defined, respectively, as
$$ 
D_s h(r)= h(rs^{-1}) \quad (r,s\in\QQ_+); \quad
S_\beta h(\alpha) = h(\alpha-\beta) \quad (\alpha,\beta\in\ZZ(\infty)).
$$ 
The Bohr transformation transforms both shifts into the multiplication 
$M_{\zeta^\alpha}$ by the corresponding character, 
$M_{\zeta^\alpha}f=\zeta^\alpha f(\zeta)$, so that the following lemma 
holds.
 
\begin{Lemma}
$$
BD_s = M_{\zeta^{\alpha}}B \quad (\alpha = \alpha(s)) \text{ and } 
BS_\beta = M_{\zeta^\beta}B. \bull
$$
\end{Lemma}
 
\begin{Subsection}{Domain where the Bohr transform exists}
For functions having support in $\NN\subset\QQ_+$, the Bohr 
transform of $x\in X(\NN)$ is naturally defined on the corresponding 
$X'$-\emph{circular domain} $\DD^\infty_{X'}$, 
$$ 
\DD^\infty_{X'} = \big \{ \lambda = (\lambda_s)_{s\geq 1} : 
(\lambda^{\alpha(n)})_{n\geq 1}\in X(\NN)'\big \},
$$ 
or, in an additive writing, for $x\in X(\ZZ_+(\infty))$, 
$$ 
\DD^\infty_{X'} = \big\{ \lambda = (\lambda_s)_{s\geq 1} : 
(\lambda^\alpha)_{\alpha\in\ZZ_+ (\infty)} \in X(\ZZ_+(\infty))'\big\}.
$$ 
Precisely, 
$$
Bx(\lambda) = \sum_{\alpha\in\ZZ_+(\infty)} x_\alpha \lambda^\alpha,
\quad \lambda\in\DD^\infty_{X'}, 
\quad \forall x\in X(\ZZ_+(\infty)).
$$ 
Since the above series are absolutely convergent, they represent functions 
holomorphic on the interior part of the corresponding domain. In fact, 
since $X$ and $X'$ are lattices, 
$\lambda = (\lambda_s)_{s\geq 1}\in\DD^\infty_{X'}$ and 
$|\mu_s|\leq |\lambda_s|$ ($\forall s\geq 1$) entail 
$\mu\in\DD^\infty_{X'}$, so that --- following the standard complex 
analysis terminology --- $\DD^\infty_{X'}$ is \emph{a complete 
Reinhardt domain} in $\CC^\infty$. 

We often identify a lattice $X=X(\NN)$ on $\NN$ with its 
transplantation to $\ZZ_+(\infty)=\alpha(\NN)$, always with the help 
of the above defined map 
$n\mapsto \alpha(n) = (\alpha_1,\dots\alpha_s,0,\dots)$. So, given a 
sequence space $ X$, we formally should write $X\circ\alpha^{-1}$ for its 
transplanted image on $\ZZ_+(\infty)$, 
$$ 
X\circ\alpha^{-1} = \{ (x_\alpha) : x_\alpha = x_n, \text{ if } 
\alpha = \alpha (n) \text{ and } (x_n)\in X \}.
$$ 
For short, omitting $\alpha^{-1}$, we hope that the writings like 
$X(\ZZ_+(\infty))=X(\NN)$ (formally, $=X(\NN)\circ\alpha$) would 
not lead to confusions. 

For a power series realization, we speak on 
$$ 
X_A, \text{ \emph{the space of power series} }
f = \sum_{n\geq 1}\hat{f}(n)z^n,
$$ 
with $(\hat{f}(n))_{n\geq 1}\in X = X(\NN)$, ($A$ stands for 
'analytic'). Respectively,  
$$ 
X_A(\DD^\infty_{X'}) \text{ \emph{denotes the space of power series} }
f = \sum_{\alpha\in\ZZ_+(\infty)} \hat{f}(\alpha)\lambda^\alpha,
$$ 
defined (with the absolute convergence) for $\lambda\in\DD^\infty_{X'}$ 
and $(\hat{f}(\alpha))\in X(\ZZ_+(\infty))=X\circ\alpha^{-1}$. In fact, 
the latter one is the Bohr transform of the former, 
$X_A(\DD^\infty_{X'}) = BX(\ZZ_+(\infty))=BX(\NN)$. The spaces 
$X_A(\DD_{X'}^\infty)$ are equipped with the standard $X$-norm.
\end{Subsection}

\begin{Example}
For $X=\ell^p$, $1\leq p\leq\infty$, it is easy to see that the corresponding 
domains are \emph{multidiscs}:  
$$ 
\DD^\infty_{X'} = \big\{ \lambda = (\lambda_s)_{s\geq 1} : |\lambda_s|<1; 
\lambda\in \ell^{p'} \big\} \text{ for } 1<p\leq \infty,
$$ 
(since, by the Euler product/summation formula, 
$\lambda\in \ell^{p'}(\NN)$ $\Leftrightarrow$ 
$(\lambda^\alpha)_{\alpha\in\ZZ_+(\infty)}\in\ell^{p'}(\ZZ_+(\infty))$), 
where $\frac{1}{p'}+\frac{1}{p}=1$, and for $X=\ell^1$ 
$$ 
\DD^\infty_{X'}= \overline\DD^\infty =: 
\big\{ \lambda = (\lambda_s)_{s\geq 1} : |\lambda_s|\leq 1 \big\}.
$$ 
For short, 
$$ 
\DD^\infty_{(\ell^p)'} = \DD^\infty_{p'}.
$$ 
(For $p=2$, $\DD_2^\infty$ is the \emph{Hilbert multidisc} considered 
already in 1909 in \cite{Hil1909}). Bohr transforms of $X=\ell^p$ are power 
series with $\ell^p$ coefficients  
$$ 
B\ell^p = \ell^p_A(\DD_{p'}^\infty) 
= \left\{f : f(\lambda) = \sum_{\alpha\in\ZZ_+(\infty)} a_\alpha\lambda^\alpha 
: a\in \ell^p(\ZZ_+(\infty)); \lambda\in\DD_{p'}^\infty \right\} .
$$ 
Clearly, $\DD_{p'}^\infty = \DD^\infty\cap \ell^{p'}(\NN)$ and 
$\DD_\infty^\infty=\overline\DD^\infty$. 

\emph{More general}, for an {\bf N-ISL $X$} the following inclusions hold  
$$ 
\DD_1^\infty\subset\DD_{X'}^\infty\subset\overline\DD_\infty^\infty,
$$ 
so that $\DD_{X'}^\infty$ is in between of two above multidiscs. 

The following is an immediate consequence of Lemma 4.1. 
\end{Example}
 
\begin{Lemma} 
Let $ X$ be an {\bf ISL}. A subspace $E\subset X$ is $(D_n)_{n\in\NN}$ 
invariant if and only if its Bohr image $BE$ is 
$M_\zeta =: \{M_{\zeta^\alpha} : \alpha\in\ZZ(\infty)\}$-invariant. 

In particular, $x\in X$ is $(D_n)_{n\in\NN}$-cyclic if and only if $Bx$ 
is $M_\zeta$-cyclic in $BX(\DD_{X'}^\infty)$, i.e. 
$$ 
BX(\DD_{X'}^\infty) = \clos\big(pBx : p \text{ runs over all polynomials} \big)
$$ 
($\clos$ with respect to the corresponding topology).
\end{Lemma} 
 
\begin{Comments}
{\bf (1)} 
$M_{\zeta}$-cyclic functions are also called \emph{weakly invertible}.

{\bf (2)} 
In the above notation, a semicharacter (totally multiplicative sequence) 
$x=(x_n)_{n\geq 1}$, $x_n = \omega^{\alpha(n)}$, lying in $\ell^2(\NN)$, has 
as its Bohr transform a \emph{reproducing kernel} 
$k_{\overline{\omega}}=Bx$ of the Hardy space 
$\ell^2_A(\DD_2^\infty)=H^2(\DD_2^\infty)$ at the point 
$\overline{\omega}=(\overline{\omega}_j)_{j\geq 1}\in\DD_2^\infty$ of the 
Hilbert multidisc $\DD_2^\infty$, 
$$ 
k_{\overline{\omega}}(\zeta) 
= \sum_{\alpha\in\ZZ_+(\infty)}\omega^\alpha\zeta^\alpha
= \prod_{j\in\NN}{\frac{1}{1-\zeta_j\omega_j}}, \quad \zeta\in\DD_2^\infty.
$$ 
In this case, Proposition 3.1 is classical, and at least 7 independent 
proofs are known, as in the first 7 references quoted in Section 1. In 
fact, the matter is on the cyclicity of some (various from paper to paper) 
totally multiplicative sequences in the space $\ell^2$; paper \cite{Que2015} 
deals with the space $\ell^\infty$. All these proofs, excepting 
\cite{Nik2012}, essentially use certain number theoretic identities and 
inequalities, in particular the following M\"obius inversion formula which 
holds for every complex sequences $a=(a_j)_{j\geq 1}$ and $b=(b_j)_{j\geq 1}$: 
$$ 
a_n = \sum_{d\mid n}b_d \quad (n=1,2,\dots) \Leftrightarrow  
b_n = \sum_{d\mid n} \mu\left(\frac{n}{d}\right) a_d \quad (n=1,2,\dots),
$$ 
where $\mu$ stands for the \emph{M\"obius function}: $\mu(p)=-1$, 
$\mu(p^k)=0$ for $k\geq 2$ and every prime $ p$, and 
$\mu (p_1^{\alpha_1}\dots p_s^{\alpha_s})
=\mu(p_1^{\alpha_1})\dots\mu(p_s^{\alpha_s})$ for any integer 
$n=p_1^{\alpha_1}\dots p_s^{\alpha_s}$. In fact, the mentioned M\"obius 
inversion formula has a very simple form in terms of the Bohr transforms:  
$$ 
Ba= B\one\cdot Bb \Leftrightarrow  
Bb=\frac{1}{B\one} \cdot Ba ,
$$
where $\one = (1,1,\dots)$, $B\one(\zeta)=\prod_{k\geq 1}\frac{1}{1-\zeta_k}$ 
for $\sum_k|\zeta_k| < \infty$, $|\zeta_k|<1$ ($k\geq 1$). 

The tenth proof of the cyclicity given in Sections 1 and 3 is the simplest 
one --- it uses nothing special but the definition of the cyclicity itself 
and the basic remarks on the duality. 

{\bf (3)}
In the Hilbert space case of $\ell^2(\NN)$, the Parseval theorem says that  
$$
\|x\|^2_{\ell^2(\NN)} = \int_{\TT^\infty} |Bx(\zeta)|^2\,dm_\infty(\zeta), 
\quad \forall x\in \ell^2(\NN).
$$ 
Moreover, at least for functions $f$ continuous on $\TT^\infty$ (for 
instance, for $f=Bx$, $x\in \ell^1(\NN)$), the following Bohr's formula holds 
$$ 
\int_{\TT^\infty} f(\zeta)\,dm_\infty(\zeta) 
= \lim_{T\to \infty}\frac{1}{2T} \int_{-T}^{T}f(K^t)\,dt,
$$ 
where $t\mapsto K^t := (K_1^{-it},K_2^{-it},\dots)\in \TT^\infty$, 
$t\in\RR$, stands for an arbitrary \emph{Kronecker solenoid} on 
$\TT^\infty$; the latter means the following property: the numbers 
$\log K_j$, $j= 1,2,\dots$ are \emph{rationally independent} ($K_j$ are 
real and positive). In particular, for $K_j=p_j$ (consecutive primes) and 
for $x\in \ell^2(\NN)$, we deal with \emph{Dirichlet series} 
$$
Bx(K^t) = \sum_{n\in\NN} \frac{x_n}{n^{it}}; \quad t\in \RR.
$$ 
Multiplication in $BX(\DD_{X'}^\infty)$ (in particular, multiplication of 
Dirichlet series) corresponds to the so called ``Dirichlet convolution'' in 
$X(\NN)$. The choice $x_n= n^{-\gamma}$, $\gamma>1/2$ (a completely 
multiplicative sequence!), gives \emph{the Euler-Riemann $\zeta$-function}
on a vertical line, $Bx(K^t)=\sum_{n\in \NN}\frac{1}{n^{\gamma+it}}$, 
$t\in\RR$.

{\bf (4) Finite dimensional sections} $\DD_{X'}^J=\DD_{X'}^\infty\cap\CC^J$ 
($J$ is a set of indices). Given a (finite) set of indices $J\subset\NN$, 
let $\CC^J = \{\zeta\in\CC^\infty : \zeta=(\zeta_j), \supp(\zeta)\subset J\}$. 
Obviously, 
$$ 
\ell^1_{\|\delta_\alpha\|}(\ZZ_+(\infty)) \subset X
\subset \ell^\infty_{\|\delta_\alpha\|}(\ZZ_+(\infty)),
$$ 
and hence 
$$ 
\left\{ \lambda=(\lambda_j) : \sum_{\alpha\in\ZZ_+(\infty)} 
  \frac{|\lambda^\alpha|}{\|\delta_\alpha\|_X}<\infty \right\}
\subset\DD_{X'}^\infty\subset 
\left\{ \lambda=(\lambda_j) : \sup_\alpha 
  \frac{|\lambda^\alpha|}{\|\delta_\alpha\|_X}<\infty \right\}.
$$ 
In particular, if $X$ is an {\bf N-ISL} (as any $\ell^p$ space), we simply 
have 
$$ 
\{ \lambda\in \ell^1: |\lambda_j|<1 \} \subset \DD_{X'}^\infty \subset 
\{ \lambda : |\lambda_j| \leq  1 \},
$$
and so 
$$ 
\DD^J \subset \DD_{X'}^J \subset \overline{\DD}^J 
  \text{ \emph{for every finite set} } J\subset\NN.
$$ 
In general, we can only claim $\lambda\in\DD_{X'}^\infty$ $\Rightarrow$ 
$\sup_{k\in\NN} |\lambda_j^k|/\|\delta_{ke_j}\|_X<\infty$, where 
$e_j=(\delta_{j,k})_{k\geq 1}\in\ZZ_+(\infty)$, so that 
$$ 
|\lambda_j|\leq \varliminf_k \|\delta_{ke_j}\|_X^{1/k} =: R_j < \infty
\quad (\forall j).
$$ 
Sometimes, the latter necessary condition for $\lambda\in\DD_{X'}^\infty$ 
is not far from being sufficient, as for example for $X=\ell^p_{w_n}(\NN)$  
with a totally multiplicative weight $w_n=n^\gamma$, $n\geq 1$ (in this 
case, $R_j=p_j^\gamma$, $p_j$ being the $j$-th consecutive prime). For 
some other examples this is not the case. For example, for an important 
space $X=\ell^2_{w_\alpha}(\ZZ_+(\infty))$ with 
$$ 
w_\alpha^2 = \frac{\alpha!}{|\alpha|!} 
= \frac{\alpha_1!\alpha_2!\dots\alpha_s!}{(\alpha_1+\alpha_2+\dots\alpha_s)!}, 
\quad \alpha\in\ZZ_+(\infty),
$$ 
$R_j=1$ for every $j$, but $\DD_{X'}^J$ are rather far from the polydiscs 
$\DD^J$; in fact, $\DD_{X'}^J$ are the \emph{open unit balls}
$\BB^J=\{\lambda = (\lambda_j)_{j\in J}: \sum_j|\lambda_j|^2< 1\}$ 
(see Lemma 8.1). The Bohr transform of $X$ is known as the 
\emph{Drury--Arveson space}, important for several variables interpolation 
theory and several operators theory (the row contractions), see for 
example \cite{ARS2010,RiS2012}, and references therein. We return to this 
space in Sections 6 and 8. 
\end{Comments}
 
\begin{Subsection}{Multipliers}
Given an {\bf ISL $X=X(\NN)$}, a (formal) power series 
$F=\sum_{\alpha\in\ZZ_+ (\infty)}\hat{F}(\alpha)\zeta^\alpha$ in 
$\DD_{X'}^\infty$ is called \emph{multiplier} (or $X$-\emph{multiplier}) 
if 
$$ 
f\in X_A(\DD_{X'}^\infty) \Rightarrow Ff\in X_A(\DD_{X'}^\infty),
$$ 
where $X_A(\DD_{X'}^\infty)=BX(\DD_{X'}^\infty)$. The set of multipliers, 
denoted $\Mult(X_A)=\Mult(X_A(\DD_{X'}^\infty))$ and endowed with the 
multiplier norm and pointwise multiplication, is a Banach algebra 
containing polynomials and such that 
$\Mult(X_A(\DD_{X'}^\infty))\subset 
X_A(\DD_{X'}^\infty)\cap H^\infty(\DD_{X'}^\infty)$. It is clear that
$$ 
\sum_{\alpha\in\ZZ_+(\infty)}|\hat{F}(\alpha)|\cdot\|\zeta^\alpha\|_{\Mult} < \infty 
\Rightarrow F\in\Mult(X_A),
$$ 
and $F\in\Mult(X_A)$ $\Rightarrow$ 
$\sup_\alpha|\hat{F}(\alpha)|\cdot\|\zeta^\alpha\|_\Mult<\infty$. Indeed, 
the latter property follows from the rotation invariance of $\Mult(X_A)$: 
if $F_t(\zeta)=F(t\zeta)$, where $t\zeta=(t_i\zeta_i)_{i\geq 1}$, 
$t_i\in\TT$, then $F_tf=(Ff_{\overline{t}})_t$ and hence 
$\|F_t\|_\Mult=\|F\|_\Mult$,  
$$
\|\hat{F}(\alpha)\zeta^\alpha\|_\Mult
= \left\|\int_{\TT^\infty}\overline{t}^{\alpha}F_t\,dm_\infty(t)\right\|_\Mult
\leq \|F\|_\Mult. 
$$ 
It follows, in particular, that 
$$ 
\ell^1_A(\DD_{X'}^\infty)\subset \Mult(X_A(\DD_{X'}^\infty))  
\Leftrightarrow  
\sup_\alpha \|\zeta^\alpha\|_\Mult = \sup_{n\geq 1}\|D_n\|_{X\to X}<\infty .
$$ 
The following lemma is obvious. 
 
\begin{Lemma}
Let $X$ be an ISL, and suppose $f,g\in X_A$ are cyclic and at least one of 
them is in $\Mult(X_A)$. Then $F=fg$ is cyclic. 
\end{Lemma}
 
\begin{proof}
Suppose $g\in\Mult(X_A)$, and let $(p_n)$ be polynomials such that 
$\lim_n\|p_nf-1\|_X=0$. Then $\lim_np_nF=g\in E_F$, and since $g$ is 
cyclic, $E_F=X_A$.
\end{proof}
\end{Subsection}
 
\section{Cyclicity of (relatively) multiplicative sequences}
 
A sequence $x=(x_k)_{k\in \NN}$ (of complex numbers) is called 
\emph{(relatively) multiplicative} if $x_{nk}=x_nx_k$ for coprime $n$ and 
$k$: $\GCD(n,k)=1$. For example, the M\"obius function $\mu$ is 
multiplicative (but not totally multiplicative). A long list of arithmetic 
multiplicative functions can be found in \cite{MF-Wiki}.
 
\begin{Lemma}
(1) A sequence $x=(x_k)_{k\in\NN}$ is multiplicative if and only if, for 
every $N\in\NN$, its Bohr transform $Bx$ is of the form 
$$ 
Bx(\zeta)=f_1(\zeta_1)\dots f_N(\zeta_N)F_N(\zeta), \quad\zeta=(\zeta_1,\dots),
$$ 
where $f_j$ are functions (formal power series) in one variable and $F_N$ 
depends on $\zeta_{N+1},\zeta_{N+2},\dots$ only. The representation is 
unique.

(2) A sequence $x=(x_k)_{k\in\NN}$ with finite support is multiplicative 
if and only if, for a certain $N$ and some polynomials $f_j$, 
$$
Bx(\zeta) = f_1(\zeta_1)\dots f_N(\zeta_N), 
\zeta = (\zeta_1,\dots).
$$ 

Any function of this Lemma, say $f$, satisfies $\hat{f}(0)=1$, 
$0=(0,0,\dots)$.
\end{Lemma}

\begin{proof}
(1) If $x$ is multiplicative and $n=p_1^{\alpha_1}\dots p_s^{\alpha_s}$, 
then $x_n=a_1(\alpha_1)\dots a_s(\alpha_s)$ where $a_s=(a_s(j))_{j\geq 1}$, 
$a_s(j)=x_{p_s^j}$. Making the Euler partial summation 
$$ 
Bx(\zeta) = \sum_{\alpha_1\geq 0}a_1(\alpha_1)\zeta_1^{\alpha_1}
\sum_{\alpha_2\geq 0}a_2(\alpha_2)\zeta_2^{\alpha_2}\dots
\sum_{\alpha_{N+1}}a_{N+1}(\alpha_{N+1})\zeta_{N+1}^{\alpha_{N+1}}\dots,
$$ 
we obtain the claimed expression. The converse is also obvious. 

(2) This is an automatic consequence of (1) since $Bx$ depends on a finite 
family of variables only.
\end{proof}
 
Below, we make use the following \emph{restriction operation} (and the 
corresponding notation): given a set of indices $\sigma\subset\NN$, an 
{\bf ISL $X$}, and $x\in X$, define $R_\sigma x$ by 
$$ 
B(R_\sigma x)=\sum_{\alpha,\supp(\alpha)\subset\sigma} 
\hat{B}x(\alpha)\zeta^\alpha,
$$ 
or equivalently, $(R_\sigma x)_n=x_n$ if $\supp(\alpha(n))\subset\sigma$, 
and $(R_\sigma x)_n=0$ for any other $n\in\NN$. We also are using the 
notation
$$ 
R_\sigma F = B(R_\sigma x) \text{ for } F=Bx.
$$ 
Clearly, $R_\sigma x\in X$ and $\|R_\sigma x\|\leq \|x\|$. If a sequence 
$\zeta=(\zeta_j)_{j\in\NN}$ is considered as a function on $\NN$ and 
$\chi_\sigma$ stands for the indicator (characteristic) function of 
$\sigma$, then $R_\sigma F(\zeta)=F(\chi_\sigma\zeta)$ where $F=Bx$.

The norm $\|\cdot\|_X$ is called (relatively) \emph{multiplicative norm} 
if $\|fg\|_X = \|f\|_X\|g\|_X$ for $f\in R_\sigma X$, $g\in R_{\sigma'}X$  
with $\sigma\cap\sigma'=\emptyset$ (for sequences, the product means a 
convolution on $\ZZ_+(\infty)$). One can easily check the following lemma. 
 
\begin{Lemma}
Let $X$ be an {\bf ISL}. (1) If $x\in X$ is a multiplicative sequence, 
then the functions $f_j$ ($j\in\NN$) and $F_N$ from Lemma 5.1 satisfy 
$f_j=B(R_{\{j\}}x)$ and $F_N=B(R_\sigma x)$, where 
$\sigma=\{N+1,N+2,\dots\}$. 

(2) For $X=\ell^p$, $1\leq p\leq\infty$, the norm $\|\cdot\|_X$ is 
multiplicative; the same property holds for the weighted space $X=\ell^p_w$ 
where $w=(w_n)_{n\geq 1}$ is a (relatively) multiplicative weight.
\end{Lemma}
 
The following Proposition partially reduces the cyclicity of a 
multiplicative sequence to the cyclicity of its one-variable factors $f_j$ 
(in the sense of Lemma 5.1).
 
\begin{Proposition}
Let $X$ be an {\bf ISL}. 
(1) If $x\in X$ is $D_n$-cyclic then $R_\sigma x$ is $D_n$-cyctic in 
$R_\sigma X$ for every $\sigma\subset\NN$, $\sigma\neq\emptyset$. 

In particular, if $x$ is cyclic and multiplicative, the functions $f_j$ of 
Lemma 5.1 are $z$-cyclic in  
$$
X_j(\DD) =: BR_{\{j\}}X 
= \left\{ f=\sum_{k\geq 0}a_kz^k : a_k= y_{p_j^k}, y\in X\right\},
$$ 
$X_j$ is equipped with the norm induced from $X$. 

(2) Assume the norm $\|\cdot\|_X$ is multiplicative. Then the converse to 
(1) holds: if $x\in X$ is multiplicative and every $f_j$ of Lemma 5.1 is 
$z$-cyclic in $X_j$ then $x$ is $D_n$-cyclic in $X$.

(3) If $x\in X$ is multiplicative and all $f_j$ and $F_N$ of Lemma 5.1 are 
in $\Mult(X)$ (see 4.5 above), then the converse to (1) holds: 
$z$-cyclicity of $f_j$ in $X_j$ implies the $D_n$-cyclicity of $x$ in $X$. 
\end{Proposition}
 
\begin{proof}
(1) The restriction map $R_\sigma$ is multiplicative, 
$R_\sigma(FQ)=R_\sigma(F)R_\sigma(Q)$, and if $Q$ is a polynomial in 
$\zeta$, $R_\sigma Q$ is a polynomial in $\zeta_j$, $j\in\sigma$. If 
polynomials $Q_n$ are such that $\lim_n\|FQ_n-1\|_X=0$, $F=Bx$, then 
$\lim_n \|R_\sigma(F)R_\sigma(Q_n)-1\|_X=0$, and the principal claim 
follows.

Since $f_j=R_{\{j\}}F$, $F=Bx$, the stated cyclicity of $f_j$ is also proved. 

(2) In the notation of Lemma 5.1, 
$F(\zeta)=f_1(\zeta_1)\dots f_N(\zeta_N)F_N(\zeta)$, $F=Bx$, let 
$(p_{1,n})$ be a sequence of polynomials in $\zeta_1$ satisfying 
$\lim_n\|p_{1,n}f_1-1\|_X=0$. Then 
$$ 
\lim_n \left\|p_{1,n}F-\left(\prod_2^N f_j\right)F_N\right\|_X
= \lim_n \|p_{1,n}f_1-1\|_X \left\|\left(\prod_2^Nf_j\right)F_N\right\|_X = 0,
$$ 
and so $(\prod_2^Nf_j)F_N\in BE_x$. Repeating with other variables, we 
obtain $(\prod_m^Nf_j)F_N\in BE_x$ for all $m\leq N$, and hence 
$F_N\in BE_x$ for all $N$. 

Moreover, since $F_N=BR_\sigma x$, $\sigma=\{N+1,N+2,\dots\}$ and $X$ is 
a lattice, we have $\|F\|_X\geq\|F_N\|_X$. But all Fourier--Taylor 
coefficients $\hat{F}_N(\alpha)$ of $F_N$, excepting $\hat{F}_N(0)=1$, 
become zero starting from a certain $N$, so, by Lemma 2.1(1), 
$(\sigma)\lim_NF_N=1\in BE_x$. This implies the cyclicity of $x$: $BE_x=BX$. 

(3) The same reasoning as in (2), but with a use of the multiplier norm, 
gives  
$$
\lim_n \left\|p_{1,n}F-\left( \prod_2^Nf_j\right)F_N\right\|_X
\leq \lim_n \|p_{1,n}f_1-1\|_X 
  \left\|\left(\prod_2^Nf_j\right)F_N\right\|_{\Mult(X)}=0.
$$ 
This implies $F_N\in E_F$. The weak convergence $(\sigma)\lim_NF_N=1$ 
follows by the same reason as in (2) above.
\end{proof}
 
\begin{Comments}
{\bf (1)} 
Historically, the case $X= \ell^2$ was considered earlier at least twice. In 
this case $X_j(\DD)=H^2(\DD)$ (the standard Hardy space in the disc $\DD$) 
for every $j$, the norm of $X$ is multiplicative in the sense of above 
definition, and $\Mult(\ell^2)= H^\infty(\DD)$ (the algebra of bounded 
holomorphic functions in $\DD$). The main result of \cite{Har1947} shows 
the equivalence of $D_n$-cyclicity of multiplicative sequences to the 
cyclicity of functions $f_j$ in $H^2(\DD)$, but in a slightly weaker form 
than in points (2)--(3) above, assuming additionally the Wiener algebra 
condition $f_j,F_N\in \ell^1_A$. Our proof of Proposition 5.3 is much less 
involved than the Hartman's one. Paper \cite{Koz1948b} claims the 
discussed equivalence for $X=\ell^2$ in full generality (as it follows from 
5.3(2)), but the paper contains no proofs (and since then, they never 
appeared elsewhere; the paper also contains a number of misprints).  

{\bf (2)} 
Proposition 5.3 gives yet another proof to 3.1, the eleventh one\dots (but 
under the hypotheses of 5.3(2) on the space $X$, or the hypotheses of 
5.3(3) on the sequence $x$): indeed, for completely multiplicative 
$x\in X$, the functions $f_j$ are just rational fractions 
$f_j(z)=\frac{1}{1-\lambda z}$ whose cyclicity in any {\bf ISL} is 
immediate: $(1-\lambda z)f_j=1$.  

{\bf (3)} \emph{On the cyclicity property of functions of one variable in 
$\DD$}, such as functions $f_j$ above. In general, the property is still 
not yet well studied for many lattices $X$. The classical lattice cases 
are the following:

--- \emph{the Hardy space $X=H^2$} ($f$ is cyclic if and only if it is 
(Beurling) outer);

--- \emph{the Wiener algebra $X=\ell^1_A(\DD)$} (where $f$ is cyclic if and 
only if it is invertible, i.e. $f(z)\neq 0$, 
$\forall z\in\overline{\DD}$). Of course, the latter description is still 
true for any function Banach algebra for which $ \overline{\DD}$ is the 
space of its maximal ideals. 

In principle, it is not known whether there exists a lattice (or just a 
Banach space) $X$ of holomorphic functions such that 

{\it i)} $f\in X$ $\Rightarrow$ $zf\in X$, 

{\it ii)} the disc $\DD$ is its natural definition domain, i.e., 
$$
\DD=\{\lambda \in \CC : f\mapsto f(\lambda) \text{ is bounded on } X\} ,
$$
and 

{\it iii)} a function $f\in X$ is $z$-cyclic in $X$ if and only if 
$f(z)\neq 0$ ($\forall z\in\DD$). 

\noindent
The problem was discussed, for example, already in \cite{Nik1978}. 

For some particular lattices $X_A(\DD)$, there exists quite extensive 
literature on the cyclicity problem. Without any attempt to give an 
account on the known facts, let us mention a line of research started from 
A.~Beurling paper \cite{Beu1964}. Namely, following Beurling's ideas, the 
following problem is considered in \cite{Nik1974}: given a class of 
holomorphic functions in $\DD$, to find a ``minimal'' Banach space 
$X_A(\DD)$ where they are cyclic (weakly invertible). It is shown that 
given two ``regular'' ($\log$-convex, and more) weights $0<w_n\leq W_n$, 
$n\geq 0$, a necessary and sufficient condition for the cyclicity in 
$\ell^p_A(w_n)$ of every $f\in \ell^p_A(W_n)$, $f(z)\neq 0$ ($\forall z\in\DD)$ 
is $\sum_{n\geq 1}n^{-3/2}\log\frac{W_n}{w_n}=\infty$. For more recent 
results see \cite{BEH2014}.

 Many deep \emph{bi-cyclicity} results for ``Bergman-type'' and 
``Dirichlet-type'' spaces $\ell^p(\ZZ,w_n)$ with power-like weights are 
classical (A.~Beurling, L.~Carleson, R.~Salem, D.~Newman, L.~Brown,  
A.~Shields,\dots), many others are obtained quite recently. They deal with 
the size of the boundary zero sets of functions measured with capacities, 
Hausdorff dimensions and $h$-Hausdorff measures. Just for giving some 
sample examples and sources for further references, we mention 
\cite{BHV2004,EKR2006,BEH2014,LeM2019,LeM2018} (the latter text contains 
also a comprehensive description of a large part of previous results).

Some additional information on (simple) cyclicity of polynomials (and not 
only) and a general approach to the question can be found in Subsection 
6.10(2) below.
\end{Comments}
 
\section{Beurling's necessary condition, and a few cases where it is sufficient}
 
Let $X$ be an {\bf ISL}. Following the definition of $\DD_{X'}^\infty$, 
the functional $x\mapsto Bx(\lambda)$ is linear and continuous on $X$ for 
every $\lambda\in\DD_{X'}^\infty$, so that the following lemma is obvious. 
For the case $X=\ell^2$, the lemma is usually called 'Beurling's lemma', 
\cite{Beu1945}; in fact, it was known in the 1940s independently to many: 
to A.~Wintner, A.~Beurling, Ph.~Hartman, D.~Bourgin, V.~Ya.~Kozlov\dots 
 
\begin{Lemma}
If $x\in X$ is a $D_n$-cyclic element, then $Bx(\lambda)\neq 0$ for every 
$\lambda\in\DD_{X'}^\infty$.
\end{Lemma}
 
Even if the converse to the property of Lemma 6.1 is probably never true 
in the Banach space setting (see a remark in 5.4(3)), for some special 
elements $x\in X$, it however can be valid. Below, we show a few examples 
of such $ x$'s in case of some {\bf ISL}s $X$ of interest. 
 
\begin{Subsection}{$D_n$-cyclicity in $\ell^1(\NN)=\ell^1(\ZZ_+(\infty))$}
It is well-known and easy to check that 
$\ell^1_A(\overline{\DD}^\infty)=B\ell^1(\NN)$ is a Banach algebra (with the 
usual pointwise multiplication, which corresponds to the convolution of 
sequences on the additive semigroup $\ZZ_+(\infty)$). Its maximal ideals 
are defined by the point evaluation functionals on the closed infinite 
dimensional polydisc: $f\mapsto f(\lambda)$, 
$\lambda\in\overline{\DD}^\infty=\DD_{(\ell^1)'}^\infty$. This implies the 
following fact.
\end{Subsection}
 
\begin{Lemma}
An element $x\in \ell^1(\NN)$ is $D_n$-cyclic if and only if 
$Bx(\lambda)\neq0$, $\forall\lambda\in\overline{\DD}^\infty$. 
\end{Lemma}

\begin{proof}
The necessity is already menioned above. For the sufficiency, assume 
$f:=Bx(\lambda)\neq 0$, $\forall\lambda\in\overline{\DD}^\infty$. Then 
$1/f\in \ell^1_A(\overline{\DD}^\infty)$, and so $1/f$ is the norm limit in 
$\ell^1_A(\overline{\DD}^\infty)$ of a sequence of polynomials $(p_n)$. Since 
$p_nf\in E_{Bx}$, we have $1=\lim_n(p_nf)\in E_{Bx}$, and so 
$E_{Bx}=\ell^1_A(\overline{\DD}^\infty)$.
\end{proof}
 
\begin{Subsection}{Free term dominating functions are cyclic}
Lemma 6.3 implies that a function $f\in \ell^1_A(\overline{\DD}^\infty)$ having  
$|\hat{f}(0)|>\sum_{\alpha\in\ZZ_+(\infty),\alpha\neq0}|\hat{f}(\alpha)|$ 
is cyclic in $\ell^1_A(\overline{\DD}^\infty)$, and so in any {\bf ISL} 
containing this algebra. Considering cyclicity for the weakened topology 
$\sigma (\ell^1,c_0)$, where $c_0=c_0(\ZZ_+(\infty))$ stands for the predual 
space of sequences tending to zero at infinity, $(c_0)^*=\ell^1$, or the 
cyclicity in larger {\bf ISL $X\supset \ell^1$} satisfying the conditions of 
Lemma 2.1(2), we can include the equality case in the domination 
condition.
\end{Subsection}
 
\begin{Proposition}
The following ``free term domination condition'' 
$$ 
|\hat{f}(0)| \geq \sum_{\alpha\in\ZZ_+(\infty),\alpha\neq0}|\hat{f}(\alpha)|, 
\quad f\neq 0,
$$ 
implies the $D_n$-$\sigma(X,X')$-cyclicity of $f\neq 0$ in any {\bf ISL 
$X$} from Lemma 2.1(2) containing $\ell^1(\ZZ_+(\infty))$ (and the 
$D_n$-$\sigma (\ell^1,c_0)$-cyclicity in $\ell^1$ itself).

In particular, the conclusion holds for any polynomial $f$ with a 
dominating free term.
\end{Proposition}
 
\begin{proof}
Multiplying by a constant, we can always think that $\hat{f}(0)=1$. Let 
$0<r<1$ and $f_r=1+r(f-1)$, $0<r<1$. Then 
$1/f_r\in \ell^1_A((\overline{\DD}^\infty)$, and  
$$
\frac{f}{f_r} = 1+\frac{f-f_r}{f_r}.
$$ 
Moreover, 
$\|f-f_r\|_1=(1-r)\sum_{\alpha\in\ZZ_+(\infty),\alpha\neq 0}|\hat{f}(\alpha)| 
\leq 1-r$, and 
$$ 
f_r = 1+r(f-1), \text{ where } \|r(f-1)\|_1\leq r< 1.
$$ 
Hence, 
$$ 
\|1/f_r\|_1 \leq \frac{1}{1-r}, \text{ and }
\left\|\frac{f}{f_r}\right\|_1 \leq 1+\left\|(f-f_r)\frac{1}{f_r}\right\|_1 \leq 2 
\quad (0<r<1).
$$ 
Moreover, $f/f_r\in E_f^{\ell^1}\subset E_f^X$, and by Lemma 2.1 
$1=(\sigma)\lim_{r\to 1}\frac{f}{f_r}\in E^X_f$ (and the same for 
$\sigma(\ell^1,c_0)$ convergence in $\ell^1$), and so $E_f^X=X$, and $f$ is 
$\sigma(X,X')$-cyclic.  
\end{proof}

\begin{Comments}
{\bf (1)}
As we will see, for a special class of polynomials in $\CC^N$ (factorable 
polynomials of theorem 6.9 below), the cyclicity holds under the only 
condition $f(\zeta)\neq 0$, $\zeta\in\DD^N$ in any {\bf ISL $X_A$} 
containing $\ell^1_A$ and satisfying conditions of Lemma 2.1(2). Since every 
polynomial \emph{in one variable} is linearly factorable, we keep that 
claim for $z$-cyclicity of any polynomial in this kind {\bf ISL}'s in 
$\DD$.

As to generic \emph{polynomials in several variables}, the situation is 
more involved; we discuss it briefly in Comments 6.10 below. 

{\bf (2)} 
Proposition 6.5 is applied to $\ell_A^p$, $1<p<\infty$, and to any reflexive 
{\bf ISL $X_A\supset \ell^1_A$} where $X_{00}$ is dense. For $p=\infty$ and 
$p=1$, it concerns the cyclicity for weak$^{*}$ topologies 
$\sigma(\ell^\infty,\ell^1)$, resp. $\sigma(\ell^1,c_0)$ (but not 
$\sigma(\ell^1,\ell^\infty)$; look on $f=1-z$). 
\end{Comments}
 
\begin{Subsection}{Linearly factorable polynomials are cyclic if 
$\DD_{X'}^N=\DD^N$}
Whatever is an {\bf ISL $X$}, it contains the linear subset $X_{00}$ 
consisting of all finitely supported sequences, and its Bohr transform 
$BX_{00}$ is the set of all polynomials in $\CC^\infty$. Every polynomial 
$p=Bx$, $x\in X_{00}$, $p(\zeta)=p(\zeta_1,\zeta_2,\dots)
=\sum_{n=1}^Nx_n\zeta^{\alpha(n)}$, depends on a finite number of 
variables $\zeta_j$, and one can consider the minimal set of such 
essential variables 
$$
J(x)=\bigcup_{n,x_n\neq 0}\supp(\alpha(n)).
$$ 
Given a subset of indices $J\subset\NN$, consider the restriction $BX^J$ 
of the {\bf ISL $BX$} to the variables $\zeta_j$, $j\in J$, namely, 
$$ 
BX^J = \left\{ f\in BX : 
f = \!\!\!\!\!\sum_{\alpha\in\ZZ_+(\infty)}\!\!\!\!\hat{f}(\alpha)\zeta^\alpha, 
\supp(\hat{f}(\alpha))\subset (\alpha\in\ZZ_+(\infty): 
\supp(\alpha)\subset J)\right\}.
$$ 
It was observed in \cite{Nik2012}, that for $X=\ell^2$, a polynomial $Bx$ is 
cyclic in $B\ell^2$ (which is the Hardy space on $\DD_2^\infty$) if and only 
if it is (already) cyclic in $(B\ell^2)^J$, where $J$ is any (finite) subset 
of $\NN$, $J(x)\subset J$; if $J=[1,N]$, $(B\ell^2)^J=H^2(\DD^N)$. The same 
reasoning as in \cite{Nik2012} shows the following lemma.
\end{Subsection}
 
\begin{Lemma}
A polynomial $Bx$ is $M_\zeta$ cyclic in $BX$ if and only if it is 
(already) cyclic in $(BX)^J$, where $J$ is any (finite) subset of $\NN$, 
$J(x)\subset J$.
\end{Lemma}
 
Applying Lemma 6.8 with Lemma 6.1, we see that if a polynomial $p=Bx$ is 
cyclic in $(BX)^J$, then $p(\zeta)\neq 0$ for $\zeta\in\DD_{X'}^J$, where  
$$ 
\DD_{X'}^J := \DD^\infty_{X'} \cap \CC^J.
$$ 
If $X$ is an {\bf N-ISL}, for every finite set $J$,
$\DD^J\subset\DD_{X'}^J\subset\overline{\DD}^J$ (but for a weighted 
$X_{w_n}$ these finite dimensional sections can be different from the 
polydiscs).

A ``polynomial converse'' to above Beurling's necessary condition 6.1 claims 
that, given a polynomial $p=Bx$, $(Bx(\lambda)\neq 0$, 
$\forall\lambda\in\DD_{X'}^J)$ $\Rightarrow$ ($x$ is $D_n$-cyclic). 
Historically, this property was first claimed \emph{for the Hilbert space} 
$X=\ell^2(\NN)$ (for which $\DD_{X'}^{[1,N]}=\DD^N$) by V.~Ya.~Kozlov in 
\cite{Koz1948a}; however, the first published proof appeared only 20 years 
later, in \cite{NGN1970} of J.~Neuwirth, J.~Ginsberg, and D.~Newman. For a 
simplified proof and generalizations, see \cite{Nik2012}, \cite{Nik2019} 
(also, \cite{Nik2018b} for details of Kozlov's personal story).

It follows (of course) that the condition $p(\lambda)\neq 0$, 
$\forall\lambda\in\DD^J$ is still necessary and sufficient for cyclicity 
in every larger {\bf ISL $X_A\supset \ell_A^2$}.

In the next Proposition, we consider {\bf ISL}'s $X$ containing $\ell^1$, so 
that  
$$ 
\DD_{X'}^{[1,N]}\subset\overline{\DD}^N, \quad \forall N\geq 1,
$$ 
and will restrict ourselves to {\bf ISL}'s $X= X(\ZZ_+^N)$ on $\ZZ_+^N$. 
The Bohr transforms $Bx$ of finitely supported sequences $x\in X_{00}$ 
(i.e., polynomials), in principle, are defined on the whole $\CC^N$; as 
above we use the notation  
$$ 
X_A(\DD^N):= BX(\ZZ_+^N).
$$ 
A polynomial $f$ in $\CC^N$ is said \emph{linearly factorable} if 
$$ 
f=\prod_{j=1}^nf_j,
$$ 
where $f_j$ are affine linear functions, 
$f_j(\zeta)=f_j(\zeta_1,\dots,\zeta_N) := a + \sum_{i=1}^Na_i\zeta_i$ 
($a,a_i\in\CC$). If $f$ is linearly factorable then the zero set of $f$ 
in $\CC^N$ is a finite union of (complex) hyperplanes. For $N=1$, every 
polynomial is linearly factorable, but for $N\geq2$ this is already not 
the case (for example, 
$f(\zeta)=h(\zeta_1,\dots\zeta_{N-1})+\lambda\zeta_N$is not factorable, if 
$\lambda\neq0$, $h$ a polynomial of degree $\geq2$). 
 
\begin{Proposition}
Let $X(\ZZ_+^N)$ be an {\bf ISL} containing $\ell^1(\ZZ_+^N)$, and $f$ a 
linearly factorable polynomial in $\CC^N$, such that $f(\zeta)\neq 0$ 
for $\zeta\in\DD^N$. Then $f$ is cyclic in $X_A(\DD^N)$ (for weak 
$\sigma(X,X')$ topology).

In particular, for $N=1$, every polynomial with $f(\zeta)\neq 0$, 
$\zeta\in\DD$ is cyclic in any {\bf ISL} of the above type (for example, 
it is norm cyclic in any $\ell_A^p(\DD)$, $1<p<\infty$, and, generally, in 
any reflexive {\bf ISL} containing $\ell^1$ (think on $X=\ell^p(\ell^q)$ type 
spaces)).  
\end{Proposition}

\begin{proof}
Let $f=\prod_{j=1}^nf_j$, where $ f_j$ are affine linear in $\CC^N$. With 
Lemma 4.6, it suffices to show that every $f_j$ is cyclic in $X_A$. Since 
$f_j$ is an affine linear polynomial, 
$f_j(\zeta)=a+\sum_{i=1}^Na_i\zeta_i$, and $f_j(\zeta)\neq 0$ for 
$\zeta\in\DD^N$, it follows $|a|\geq\sum_{i=1}^N|a_i|$. So $f_j$ is 
$(\sigma)$-cyclic in $ X$ by Proposition 6.5.
\end{proof}

\begin{Rem} 
We return later on (Section 8) to the cyclicity of factorable polynomials 
in some {\bf ISL}'s $X(\ZZ_+^N)$ with $\DD_{X'}^N=\BB^N$ (the unit ball in 
$\CC^N$).
\end{Rem} 
 
\begin{Comments}
{\bf (1)}
By the same reason as in the last lines, one can see that \emph{a sequence 
$x=(x_k)\in X$ supported by the set of prime integers and such that 
$Bx(\zeta)\neq 0$ for $\zeta\in\DD^N$, $\forall N\geq 1$ satisfies 
$|x_1|\geq\sum_{k\geq 2}|x_k|$, and hence $x$ is $D_n$-cyclic in $X$, an 
{\bf ISL} from Proposition} 6.5.
\end{Comments}

{\bf (2)} \emph{Cyclicity of polynomials in a single variable}. This 
elementary case plays a special role due to Proposition 5.3. Here, some 
(folklore) details are presented since it appears impossible to find 
written references. 

Let $X$ be a shift invariant sequence Banach space,  
$$ 
X_A = \left\{ f = \sum_{n\geq 0}\hat{f}(n)z^n, 
  (\hat{f}(n))_{n\geq 0}\in X\right\}, 
$$ 
where the set of all polynomials is contained and dense in $X_A$. Keeping 
the above notation, denote 
$$ 
\DD_{X'} = \{ \lambda\in\CC : \varphi_\lambda\colon f\mapsto f(\lambda) 
\text{ is bounded on polynomials in } X_A\}
$$ 
the domain of definition of $X_A$. 
 
\begin{Claim}
Under the above hypothesis, a polynomial $f$ is $z$-cyclic in $X_A$ if and 
only if $f(\lambda)\neq 0$, $\lambda\in\DD_{X'}$. 
\end{Claim}
 
Indeed, if $f(\lambda)=0$ for some $\lambda\in\DD_{X'}$, $f$ is clearly 
non cyclic. Conversely, if $f(z)\neq 0$, $z\in\DD_{X'}$, $f$ is a finite 
product $f=c\prod(z-\lambda)$, $c\neq 0$ of factors $z-\lambda$, 
$\lambda\in\CC\setminus\DD_{X'}$ (or $f=const=c\neq 0$). Since 
$\varphi_\lambda$, $\lambda\in\CC\setminus\DD_{X'}$ is unbounded on 
polynomials, the kernel 
$\Ker(\varphi_\lambda)=\{p : p \text{ is a polynomial}, p(\lambda)= 0\}$ 
is dense in $X_A$. The $z$-cyclicity of $z-\lambda$ follows. Since 
$f=(z-\lambda)q(z)$ and the polynomial $q$ is (obviously) a multiplier of 
$X_A$, $q\in E_f$ (see Lemma 4.6). Repeating the reasoning with the 
resting zeroes of $ f$, we get $c\in E_f$ and so $E_f=X_A$. This means 
that $f$ is cyclic.  
$\hfill\square$
 
\begin{Example} 
Let $X$ be a shift invariant sequence lattice on $\ZZ_+$. Then there 
exists $R$, $0\leq R\leq\infty$, such that 
$\DD_{X'}=\{\lambda\in\CC : (\lambda^n)_{n\geq 0}\in X'\}$ is either 
$\DD_R=\{\lambda\in\CC : |\lambda|<R\}$ or 
$\overline{\DD}_R=\{\lambda\in\CC : |\lambda|\leq R\}$. If $R=\infty$, the 
only cyclic polynomial is $f=const\neq 0$. If $R=0$, every polynomial $f$ 
with $f(0)\neq 0$ is cyclic.  

If $X=\ell^p(\ZZ_+,w_n)$, $w_n>0$, we have $R=\varlimsup_n(w_n^{1/n})$ (and 
$\DD_{X'}=\overline{\DD}_R$ $\Leftrightarrow$ $(R^n/w_n)\in \ell^{p'}$, 
$1/p+1/p'=1$). In particular, for Dirichlet-type weights 
$w_n=(1+n)^\gamma$, $\DD_{X'}=\DD$ for $\gamma\leq 1/p'$, and 
$\DD_{X'}=\overline{\DD}$ for $\gamma>1/p'$. By the way, in the latter 
case, $X_A$ is a Banach algebra having $\overline{\DD}$ as the space of 
maximal ideals (see \cite{ENZ1999}), and so the cyclicity criterion is 
$f(\lambda)\neq 0$, $\lambda\in\overline{\DD}$ for any function $f\in X$, 
not only for polynomials.
\end{Example} 
 
{\bf (3)}
\emph{$z$-cyclicity problem for generic functions in one variable}. 
It is far from to be trivial and still mostly not solved, even for the 
Dirichlet-type lattices $X=\ell^p(\ZZ_+,w_n)$, $w_n=(1+n)^\gamma$, despite 
many known deep results. For these spaces, the question is in the size of 
the boundary zero set $Z(f)=\{\lambda\in\TT : f(\lambda)=0\}$ measured 
with the corresponding capacity (as it is the case for bi-cyclic sequences 
on $\ZZ$), see Comments 5.4(3). The use of the potential theory forces to 
work mostly with the Hilbert space cases $\ell^2(\ZZ_+,w_n)$. See for 
example, a recent paper \cite{KLeMZ2020} where many other researches are 
quoted. Notice that usually the $z$-cyclicity questions in the framework 
of holomorphic spaces $X_A$ are more involved than the ``real-analysis'' 
problems for bi-cyclicity on $\ZZ$, but not always --- for example, a deep 
construction of N.~Lev--A.~Olevskii (2011) of two $\ell^1(\ZZ)$-functions $f$ 
and $g$ with the same zero sets on $\TT$, $Z(f)=Z(g)$, one of which is 
$\ell^p(\ZZ)$-bicyclic ($1<p<2$), and the second is not, contrasts with a 
rather obvious such example for analytic $z$-cyclicity: $f=(1-z)^N$ is 
$z$-cyclic in all $\ell^p(\ZZ_+)_A$ ($1<p\leq\infty$), and 
$g=(1-z)^N\exp\frac{z+1}{z-1}$ is not in any $\ell^p(\ZZ_+)_A$ 
($1<p\leq\infty$) (known already from an ``abstract Keldysh approach'' of 
\cite{Nik1974}), but $Z(f)=Z(g)=\{1\}$.   

{\bf (4)} \emph{Polynomials in several variables}. 
The matter is about the $\zeta^\alpha$-cyclicity in spaces $X(\ZZ_+^N)_A$  
in $\DD^N_{X'}$ where $X$ is a lattice. Usually, $X$ is a ``Dirichlet-type'' 
space $X=\ell^p_{w_\alpha}(\ZZ_+^N)$. The ``Beurling necessary condition'' for 
$\zeta^\alpha$-cyclicity (if $F$ is cyclic then $F(\zeta)\neq 0$ for 
$\zeta\in\DD^N_{X'}$) is still the departure point. For the moment, all 
known results are either for the case $\DD^N_{X'}=\DD^N$ or 
$\DD^N_{X'}=\BB^N$, the unit ball of $\CC^N$: 
$\BB^N=\{\zeta\in\CC^N : \sum_1^N|\zeta_j|^2<1\}$. 

\emph{The case when $\DD^N_{X'}=\DD^N$}.
As is mentioned above, for the Hardy space $X_A=\ell^2(\ZZ_+^N)_A$ in the 
polydisc $\DD^N$, the sufficiency of Beurling's condition for polynomials 
($F$ does not vanish in $\DD^N \Rightarrow F$ is cyclic) was claimed by 
V.~Ya.~Kozlov \cite{Koz1948a} and proved in \cite{NGN1970} (then 
independently, in \cite{HLS1997} and \cite{Nik2012}, where two proofs are 
given, including the one based on a Lojasiewicz inequality).   

A power-like weighted Dirichlet-type space $X=\ell^p_{w_\alpha}(\ZZ_+^N)$, 
$$ 
w_\alpha=\alpha^s=:(1+\alpha_1)^{s_1}\dots(1+\alpha_N)^{s_N}, \quad 
\alpha\in\ZZ_+^N,
$$ 
is a Banach algebra with $\overline{\DD}^N$ as the maximal ideal space if 
and (for $p>1$) only if $s_j>1/p'$ for every $j$, $1\leq j\leq N$ (for 
$p=1$, it is an algebra for $s_j\geq 0$), see 7.1 below. Therefore, in 
this case, the cyclicity criterion for $F\in X_A$, 
$X=\ell^p_{w_\alpha}(\ZZ_+^N)$ is $F(\zeta)\neq 0$, 
$\zeta\in\overline{\DD}^N$ (not only for polynomials). 

For Dirichlet-type spaces which are not algebras, the question is 
practically settled for \emph{the case of bidisc $\DD^2$} ($N=2$) and 
always for $p=2$ by G.~Knese, L.~Kosinski, T.~Ransford, A.~Sola, see 
\cite{KKRS2019} and references therein. The answer depends on the set of 
boundary zeroes $Z(F)\bigcap\TT^2$ and says, roughly speaking, the 
following: a polynomial $F$, $F(\zeta)\neq 0$ ($\forall\zeta\in\DD^N$) is 
cyclic if $Z(F)\cap\TT^2$ is finite , and non cyclic if it is infinite 
(with some restrictions on the exponents $s_j$). For the \emph{case of 
tridisc $\DD^3$} ($N=3$), for the moment, the answers are already partial 
and depend on the dimension of $Z(F)\cap\TT^3$ and its ``orientation'', see 
\cite{Ber2018} where some observations on the general case of $\DD^N$, 
$N>2$ are also given (see also 7.2, 7.3 below). In particular, some 
results (with rather technical statements) depend on geometric 
characteristics as non-vanishing Gaussian curvature of certain 
submanifolds of $Z(F)\cap\TT^N$. 

\emph{The case when $\DD^N_{X'}=\BB^N$}. 
Here, a series of weighted Hilbert spaces 
$X=\fD_s=\ell^2_{w_\alpha(s)}(\ZZ_+(\infty))$, $s\in\RR$ is considered, 
all around the famous Drury-Arveson space, 
$$ 
\fD_s = \{ x=(x_\alpha)_{\alpha\in\ZZ_+(\infty)} : 
  (x_\alpha w_\alpha)\in \ell^2 \}, \quad
w_\alpha = \left(\frac{\alpha!}{|\alpha|!}\right)^{1/2}\cdot(|\alpha|+1)^s,
$$ 
where $\alpha!=\alpha_1!\alpha_2!\dots$, 
$|\alpha|=\alpha_1+\alpha_2+\dots$ ($\alpha\in\ZZ_+(\infty)$), $s\in\RR$. 
The \emph{Drury--Arveson space $X= \fD_0$} is the subject of 
numerous deep studies due to its role in multidimensional Nevanlinna--Pick 
theory, the multi-operator theory (the so called row contractions), and 
more; it is recognized as the most genuine multivariable analogue of the 
Hardy space of the disc. See \cite{RiS2012, Sha2014, Fan2019} for outlines 
of the theory. The reproducing kernel formula for Bohr transforms 
$F\in B\fD_0$,  
$$ 
F(\zeta) = (F,k_\zeta)_{\fD_0} \text{ where } 
k_\zeta(w)= \frac{1}{1-(w,\zeta)_{\CC^N}},
$$ 
shows that $\DD^N_{X'}=\BB^N$ for every $N\in\NN$, and these equalities 
still hold for all $\fD_s$. It is known \cite{RiS2012} that in two 
and three variables ($N=2,3$), the Beurling condition 
$Z(F)\cap\BB^N=\emptyset$ is sufficient for \emph{a polynomial $F$} to be 
cyclic in $B\fD_0$, but for $N\geq 4$ this is already not the case. 
In particular, $F(\zeta)=1-N^{N/2}\zeta_1\zeta_2\dots\zeta_N$, $N\geq 4$ 
are not cyclic in $B\fD_0$. 

On the other hand, it is shown, \cite{RiSu2016}, that $F\in B\fD_0$ 
\emph{and $1/F\in B\fD_0$} imply the cyclicity of $F$ in 
$B\fD_0$ (in contrast to the Bergman-type spaces $B\fD_s$, 
$s=-1$). This result contains, for example, that if 
$F\in\Mult(B\fD_0)$ and $F$ depends on a finitely many variables (in 
particular, if $F$ is a polynomial) and is separated from zero 
$|F(\zeta)|\geq c>0$ for $\zeta\in\BB^N$, then $F$ is cyclic in 
$B\fD_0$ (proved early in 2011 by S.~Cortea, E.~Sawyer, B.~Wick). 
Below, we show (Section 8) that, for every $N\geq 1$, factorable 
polynomials non-vanishing on $\BB^N$ are cyclic in $\fD_s$ for 
$s\leq 1/2$), and for $s>1/2$ the cyclicity criterion for every function 
$F$ is $F(\zeta)\neq 0$ for all $\zeta\in\overline{\BB}^N$. 

In \cite{KoV2022}, the case of polynomials in two variables in 
$B\fD_s(\BB^2)$ is considered. It is shown that for $s\leq 1/4$, a 
polynomial $F$ is cyclic if and only if $F(\zeta)\neq 0$ for 
$\zeta\in\BB^2$, and for $1/4<s\leq 1/2$ --- if and only if 
$Z(F)\cap\partial\BB^2$ is finite. For $s>1/2$, the cyclicity criterion is 
$F(\zeta)\neq 0$ for all $\zeta\in\overline{\BB}^2$.
 
\section{Factorable functions in Dirichlet-type spaces and algebras in polydiscs}
 
A kind of power-like weighted $\ell^2$ spaces are known as ``Dirichlet-type 
spaces'', referring to the classical Dirichlet space $\ell^2_w(\ZZ_+)_A$, 
$w_n=(1+n)^{1/2}$. Here and in the next Section, we give an elementary 
treatment to the problem of cyclicity, first, for factorable polynomials, 
and then to more general factorable functions in the framework of the 
Dirichlet-type spaces $\ell^p_w(\ZZ_+(\infty))_A$. It is already treated, 
in particular, in \cite{Ber2018,KKRS2019}, even for any polynomials but 
only for the Hilbert space case $p=2$ and the Dirichlet-type weights (as 
it is the case, in general, for the most part of the existing DCP 
literature). Speaking on polynomials, we can restrict ourselves on the 
sequence spaces over a finite dimensional subsemigroup $\ZZ_+^N$. A 
special situation occurs when the space $\ell^p_w(\ZZ_+(\infty))_A$ is an 
algebra with the point-wise multiplication and a well identified space of 
maximal ideals. In this Section we consider the weights $w_\alpha$ 
factored by variables (and so, oriented to the polydisc case $\DD^N$), and 
in Section 8 --- the ball oriented weights of Drury--Arveson type. 

Let $s\in\RR^N$, and let $w(s)$ be a power-like weight function on 
$\ZZ_+^N$ defined by
$$ 
w_\alpha (s) = (1+\alpha_1)^{s_1}\dots(1+\alpha_N)^{s_N}, 
\quad \alpha\in\ZZ_+^N.
$$ 
For short, the notation $\alpha^s:= w_\alpha(s)$ is used as well. Consider 
a lattice
$$ 
X_w^N = \ell^p(\ZZ_+^N,\alpha^s) 
  := \{x=(x_\alpha): (x_\alpha w_\alpha)\in \ell^p(\ZZ_+^N)\}, 
\quad 1\leq p\leq \infty,
$$ 
and its Bohr transform $BX^N_w=\ell^p(\ZZ_+^N,\alpha^s)_A$ defined on the 
domain $\DD^N_{(X^N_w)'}$. It is easy to show that for 
$X=\ell^p(\ZZ_+^N,\alpha^s)$,
$$ 
\DD^N_{X'} = \DD^N \text{ if } 1<p\leq\infty, s_j\leq\frac{1}{p'}=\frac{p-1}{p} 
\ (\forall j); \text{ or } p=1 \text{ and } s_j<0 \ (\forall j),
$$
and 
$$
\DD^N_{X'} = \overline{\DD}^N \text{ if } 1<p\leq\infty, 
s_j>\frac{1}{p'}=\frac{p-1}{p} \ (\forall j); \text{ or } p=1 \text{ and }
s_j\geq 0 \ (\forall j);
$$ 
mixed products of $\DD$'s and $\overline{\DD}$'s are also can be easily 
identified (but we will not enter in these details). In the generic 
infinite variable case 
$$ 
X_w=\ell^p(\ZZ_+(\infty),\alpha^s) 
:= \{ x=(x_\alpha): (x_\alpha w_\alpha)\in \ell^p(\ZZ_+(\infty))\}, 
\quad 1\leq p\leq\infty,
$$ 
one can similarly show that, for $ p>1$, 
$$ 
\DD^\infty_{X_w'} = \left\{ \lambda = (\lambda^\alpha) : 
  \sum_{j\geq 1}f_j(|\lambda_j|^{p'})<\infty \right\}, \text{ where }
f_j(x) = \sum_{n\geq 1}\frac{x^n}{(n+1)^{p's_j}} .
$$ 
In particular, (1) always, $\DD^\infty_{X_w'}\subset\overline{\DD}^\infty$; 
(2) if $s_j\leq\frac{1}{p'}$ ($\forall j$) then 
$\DD^\infty_{X_w'}=\DD^\infty_{p'}$ (i.e., $\lambda\in\DD^\infty_{X_w'}$
$\Leftrightarrow$ $|\lambda_j|<1$ ($\forall j$) and $\lambda\in \ell^{p'}(\NN)$); 
(3) for the case where $1<c\leq p's_j\leq C<\infty$ ($\forall j$), one has 
$\lambda\in\DD^\infty_{X_w'}$ $\Leftrightarrow$ $|\lambda_j|\leq 1$ 
($\forall j$) and $\lambda \in \ell^{p'}(\NN)$); and finally, 
(4) in order to have simply $\DD^\infty_{X_w'}=\overline{\DD}^\infty$, it 
is sufficient that $p's_j\geq c\log(j+1)$, $c>1/\log 2$. 

For $p=1$, we always have $\DD^\infty_{X_w'}= \overline{\DD}^\infty$.

The Hilbert space case $p=2$ in the bidisc $\DD^2$ ($N=2$) of theorem 7.2 
below is contained in \cite{KKRS2019}, and already for arbitrary 
polynomials (not only for the factorable ones) but with more involved 
proofs. See also comments in 7.3. 

We start with the following known (at least for sequences on the additive 
semigroup $\ZZ_+$ related to holomorphic spaces in $\DD$) general lemma on 
weighted $\ell^p$ spaces which are convolution algebras. In fact, the 
question has quite a long history: it started with J.~Wermer's result 
\cite{Wer1954} for the continuous case of convolution algebras 
$L^p(\RR,w)$. It could be a sample theorem for repeating everywhere, but 
was overlooked and rediscovered independently, and in a different 
language, for the sequence spaces $\ell^p(\ZZ_+,w_n)$ much later in 
\cite{Nik1968,Nik1970,KL1973,Shi1974}. It is also known that the 
sufficient condition of Lemma 7.1 for a weighted $\ell^p$ space to be a 
convolution algebra is also necessary under some natural regularity 
conditions on the weight (a kind of $\log$ concavity, but not in the full 
generality \cite{KL1973}; see for example, comments in \cite{ENZ1999}). It 
is clear that a similar to 7.1 claim can be shown for weighted $L^p$ 
norms on any locally compact abelian group (or semigroup). For 
completeness, we give a short proof to Lemma 7.1.
 
\begin{Lemma}
Let $X=\ell^p(\ZZ_+(\infty),w_\alpha)$. If 
$$ 
C_\gamma =: \sup_{\gamma\in\ZZ_+(\infty)}W_\gamma<\infty, \text{ where }
W_\gamma = \left(\sum_{\alpha+\beta=\gamma} 
  \left(\frac{w_\gamma}{w_\alpha w_\beta}\right)^{p'}\right)^{1/p'},
$$ 
$\frac{1}{p}+\frac{1}{p'}=1$, then $X_A$ is an algebra with pointwise 
multiplication, and the complex homomorphisms of $X_A$ are of the form 
$\varphi_\lambda: F\mapsto F(\lambda)$, $F\in X_A$ where 
$\lambda\in\DD^\infty_{X'}$ (in this case, the domain $\DD^\infty_{X'}$is 
compact with respect to a norm bounded coordinate-wise convergence in 
$X'=\ell^{p'}(\ZZ_+(\infty),1/w_\alpha)$). 
\end{Lemma}
 
\begin{proof}
Let $1<p<\infty$ (modifications for $p=1$ or $p=\infty$ will be obvious) 
and $\gamma\in\ZZ_+(\infty)$. It suffices to have an estimate 
$\|FG\|\leq C\|F\|\cdot\|G\|$ for the product of two polynomials 
$F,G\in X_A$. For $d=\deg(FG)$), $H=FG$ and 
$\hat{H}(\gamma)=\sum_{\alpha+\beta=\gamma}\hat{F}(\alpha)\hat{G}(\beta)$, 
one has 
\begin{align*} 
|\hat{H}(\gamma)|^p w_\gamma^p 
&= \left|\sum_{\alpha+\beta=\gamma} \hat{F}(\alpha)w_\alpha\hat{G}(\beta)
  w_\beta\cdot\frac{1}{w_\alpha w_\beta} \right|^p w_\gamma^p \leq 
\\ 
&\leq \left( \sum_{\alpha+\beta=\gamma} |\hat{F}(\alpha)w_\alpha|^p
  |\hat{G}(\beta)w_\beta|^p\right)\left(\sum_{\alpha+\beta=\gamma}
  \left(\frac{w_\gamma}{w_\alpha w_\beta}\right)^{p'}\right)^{p/p'} .
\end{align*} 
Hence, 
\begin{align*} 
\|H\|^p &= \sum_{\gamma\in\ZZ_+^N} |\hat{H}(\gamma)|^p w_\gamma^p
\leq C_\gamma^p \sum_{\gamma\in\ZZ_+^N}
  \left( \sum_{\alpha+\beta=\gamma} |\hat{F}(\alpha)w_\alpha|^p
  |\hat{G}(\beta)w_\beta|^p\right) =
\\ 
&= C_\gamma^p \left(\sum_{\alpha\in\ZZ_+^N} |\hat{F}(\alpha)w_\alpha|^p\right)
  \left(\sum_{\beta\in \ZZ_+^N} |\hat{G}(\beta)w_\beta|^p\right) 
= C_\gamma^p \|F\|^p\|G\|^p ,
\end{align*} 
with a constant $C<\infty$ independent of $F,G$. It follows that $X_A$ is 
a Banach algebra (with an equivalent norm). If $\varphi:X_A\to\CC$ is a 
homomorphism of $X_A$, and $\lambda_j=\varphi (F_j)$, where 
$F_j(\zeta)=\zeta_j$ ($\forall\zeta\in\DD_{X'}$), then 
$\varphi(\zeta^\alpha)=\lambda^\alpha$ for every $\alpha\in\ZZ_+(\infty)$, 
where $\lambda=\lambda_\varphi=:(\lambda_1,\lambda_2,\dots)$, and so 
$\varphi(F)=\varphi_\lambda(F)$ for every polynomial $F$. Since $\varphi$ 
is bounded, $\lambda\in\DD^\infty_{X'}$ and \emph{vice-versa}, every such 
$\lambda$ defines a homomorphism $\varphi_\lambda$. Clearly, 
$\varphi\mapsto\lambda_j$ is continuous for $\sigma(X',X)$ topology, and 
since the space $\{\varphi\}$ is compact, $\DD^\infty_{X'}$ is 
homeomorphic to the space of all $\{\varphi\}$ (the Gelfand space of 
$X_A$). 
\end{proof}
 
\begin{Proposition}
Let $X=\ell^p(\ZZ_+^N,\alpha^s)$. 

(1) If $s_j>1/p'$ ($\forall j$), then $X_A$ is an algebra with pointwise 
multiplication, and the cyclicity criterion for any function $F\in X_A$ is 
$F(\zeta)\neq 0$, $\forall\zeta\in\overline{\DD}^N$. 

(2) If $|s|=\sum_{j=1}^Ns_j\leq 1/p'$, and $f$ is a linearly factorable 
polynomial, then $f$ is $\sigma(X,X')$-cyclic in $X_A$ if and only if 
$f(\zeta)\neq 0$ for $\zeta\in \DD^N$. 
\end{Proposition}
 
\begin{proof} 
{\it (1)}
It suffices to check the condition of Lemma 7.1. Let $1<p<\infty$ 
(modifications for $p=1$ or $p=\infty$ are obvious) and 
$\gamma\in\ZZ_+^N$. Then (with $d=|\gamma|$), 
\begin{align*} 
\sum_{\alpha+\beta=\gamma}\left(\frac{w_\gamma}{w_\alpha w_\beta}\right)^{p'}
&= \sum_{\alpha+\beta=\gamma} 
\frac{(1+\gamma_1)^{s_1 p'}\dots(1+\gamma_N)^{s_N p'}}%
{(1+\alpha_1)^{s_1 p'}(1+\beta_1)^{s_1 p'}\dots} \leq 
\intertext{(summation on $1\leq\alpha_j\leq d-1$)}
&\leq\left( B\sum_{\alpha_1}\frac{d^{s_1p'}}{\alpha_1^{s_1p'}(d-\alpha_1)^{s_1p'}} \right)
\dots\left( B\sum_{\alpha_N}\frac{d^{s_Np'}}{\alpha_N^{s_{N1}p'}(d-\alpha_N)^{s_Np'}}\right)
\leq
\\
&\leq \left(B\sum^{d-1}_{k=1}\frac{d^m}{k^m(d-k)^m}\right)^N\leq C^{p'},
\end{align*} 
where $m=\min_j(s_jp')>1$ and $B=B(s)<\infty$ is a constant, and 
\begin{align*} 
\left( \sum^{d-1}_{k=1}\frac{d^m}{k^m(d-k)^m}\right)^N
&\leq \left(2\sum^{d/2}_{k=1}\frac{d^m}{k^m(d-k)^m}\right)^N
\leq 
\\
&\leq \left( B\sum^{d/2}_{k=1}\frac{1}{k^m}\right)^N
\leq \left( B\sum^\infty_{k=1}\frac{1}{k^m}\right)^N =: C^{p'} .
\end{align*} 
It follows that $X_A$ is a Banach algebra (with an equivalent norm), and 
its maximal ideal space is $\overline{\DD}^N$ (Lemma 7.1). The cyclicity 
claim also follows. 

{\it (2)} 
Let $f=\prod_{j=1}^nf_j$, where $f_j$ are affine linear in $\CC^N$. As in 
Proposition 6.9, it suffices to show that every $f_j$ is cyclic in $X_A$. 
For that, we repeat the reasoning of 6.5, with a different estimate. Since 
$f_j$ is an affine linear polynomial, 
$f_j(\zeta)=a-\sum_{i=1}^Na_i\zeta_i=:a-l(\zeta)$, and $f_j(\zeta)\neq 0$ 
for $\zeta\in\DD^N$, it follows $|a|\geq \sum_{i=1}^N|a_i|$. 

Simplifying the notation, let $a= 1$, $f(\zeta)=1-l(\zeta)$, 
$l(\zeta)=\sum_{i=1}^Na_i\zeta_i$, $1\geq\sum_{i=1}^N|a_i|$ and 
$$ 
f_r= 1-rl(\zeta).
$$ 
Then (pointwise in $\DD^N$), 
$$ 
\frac{1}{f_r} = \sum_{k\geq 0}r^kl(\zeta)^k
= \sum_{k\geq 0}r^k \sum_{|\alpha|=k} a^\alpha\zeta^\alpha
  \frac{|\alpha|!}{\alpha!},
$$ 
where $a=(a_1,\dots,a_N)$, $\alpha\in\ZZ_+^N$, 
$\alpha!=\alpha_1!\dots\alpha_N!$, and hence 
$$ 
\left\| \frac{1}{f_r} \right\|^p 
= \sum_{k\geq 0}r^{kp} \left\| \sum_{|\alpha|=k}a^\alpha\zeta^\alpha
  \frac{|\alpha|!}{\alpha!}\right\|^p,
$$ 
where 
$$ 
\left\| \sum_{|\alpha|=k}a^\alpha\zeta^\alpha\frac{|\alpha|!}{\alpha!}\right\|^p
= \sum_{|\alpha|=k} |a^\alpha|^p\left(\frac{|\alpha|!}{\alpha!}\right)^p w_\alpha^p.
$$ 
Since $M_k=:\max_\alpha w_\alpha^p=\max_\alpha\big(\prod_{j=1}^N(1+\alpha_j)^{s_j p}\big)$
over all $1+\alpha_j\geq 0$ having $\sum_{j=1}^N(1+\alpha_j)=N+k$ is 
$$ 
M_k = \prod_{j=1}^N \left((k+N)\frac{s_j}{|s|}\right)^{s_j p} = O(k^{|s|p})
\text{ as } k\to\infty,
$$ 
and 
$$
\sum_{|\alpha|=k}|a^\alpha|\frac{|\alpha|!}{\alpha!} 
= \left(\sum_{i=1}^N|a_i|\right)^k\leq 1,
$$ 
we get 
$$ 
\left\|\frac{1}{f_r}\right\|^p\leq C \sum_{k\geq 0}r^{kp}k^{|s|p}, \ 0<r<1.
$$ 
Hence 
$$ 
\left\|\frac{1}{f_r}\right\|^p = O\left(\frac{1}{(1-r^p)^{1+|s p}}\right), 
\quad
\left\|\frac{1}{f_r}\right\| = O\left(\frac{1}{(1-r)^{1/p+|s|}}\right)
\quad (\text{as } r\to 1),
$$ 
and so $\|1/f_r\|=O\big(\frac{1}{1-r}\big)$. Getting $f/f_r\in E_f$ and 
$\|f/f_r\|=O(1)$ (as $r\to 1$), we finish the proof with the same argument 
as in 6.5 above.
\end{proof}
 
\begin{Comments}
If a weighted {\bf ISL $X=\ell^p(\ZZ_+^N,w_\alpha)$} satisfies 
$$ 
\max\{w_\alpha: |\alpha|=k \} = O(k^S), \ S\geq 0\ (\text{as } k\to\infty)
$$ 
and $S\leq 1/p'$, the same reasoning as in 7.2(2) shows that a factorable 
polynomial $f$ is cyclic in $X_A$ if $f(\zeta)\neq 0$, $\zeta\in\DD^N$. 

It is also of interests to compare Proposition 7.2 with some results in 
\cite{Ber2018}, where it is shown (Theorem 1 of that paper) that if $p=2$ 
and the zero set $Z(f)$ of a polynomial $f$ is contained in the union of 
\emph{coordinate hyperplanes $\zeta_j=c_j$}, then the cyclicity holds 
under the weaker condition $0<s_j\leq 1/2$ ($\forall j$). However, for a 
similar assumptions but with \emph{non-coordinate hyperplanes}, this is 
already not the case, and the cyclicity is guaranteed only for the tridisc 
case $\DD^3$, $N=3$ and under condition $|s|\leq 1/2$ (coinciding with 
condition 7.2(2) for this case), see Theorem 5 of the paper cited.

In fact, under certain convergence assumptions, part 7.2(2) of the above 
Proposition can be done in the infinite variables setting of 
$X=\ell^p(\ZZ_+(\infty),w_\alpha)$, as follows. 
\end{Comments}
 
\begin{Proposition}
Let $X=\ell^p(\ZZ_+(\infty),w_\alpha(s))$, $s=(s_1,\dots)\in\RR_+^\infty$, 
and 
$$ 
w_\alpha(s) = \prod_{j=1}^\infty(1+\alpha_j)^{s_j}, \alpha\in\ZZ_+(\infty).
$$ 
Then, 

(1) $X=\ell^p(\ZZ_+(\infty),w_\alpha(s))$ is an {\bf ISL}, and an elementary 
dilation $D_{e_j}$, $e_j=(\delta_{j,k})_{k\geq 1}$ has the norm 
$\|D_{e_j}\| \leq 2^{s_j}$. The finite dimensional sections of the domain 
$\DD^\infty_{X'}$ are the polydiscs ($\DD^N_{X'}$ are either $\DD^N$ or 
$\overline{\DD}^N$, see the beginning of Section 7).

(2) If a linear affine function 
$$ 
f(\zeta) = a+\sum_{k\geq 1}a_k\zeta_k
$$ 
is nonvanishing on every $\DD^N_{X'}$, $N= 1,2,\dots$, then 
$|a|\geq\sum_{k\geq 1}|a_k|$. 

(3) Assume $|s|:=\sum_{j\geq 1}s_j<\infty$ and 
$f\in X_A\cap \ell^1(\ZZ_+(\infty))_A$ be a factorable function in the sense 
that $f=\prod_{j\geq 1}f_j$, $f_j$ being linear affine functions, and 
$$ 
0 < \prod_{j\geq 1} \|f_j\|_{\ell^1} < \infty.
$$ 
If $|s|\leq 1/p'=\frac{p-1}{p}$ and $f(\zeta)\neq 0$ for all 
$\zeta\in\DD_{X'}$, then $f$ is $\sigma(X,X')$-cyclic in $X_A$ 
($\sigma(\ell^1,c_0)$-cyclic, if $p=1$). 
\end{Proposition}
 
\begin{proof} 
Let $1<p<\infty$ (modifications for $p=1$ or $p=\infty$ are standard). 

{\it (1)}
Since $\|D_\alpha\| = \sup\{\frac{w_{\alpha+\beta}}{w_\beta} : \beta\in\ZZ_+(\infty)\}$, 
we have for $\alpha=e_j$: 
\begin{align*} 
\|D_{e_j}\| &= \|M_{\zeta_j}\| 
= \sup\left\{\frac{w_{e_j+\beta}}{w_\beta} : \beta\in\ZZ_+(\infty)\right\}
\\
&= \sup\left\{\frac{(1+\beta_j+1)^{s_j}}{(1+\beta_j)^{s_j}} : 
  \beta\in\ZZ_+(\infty)\right\}\leq 2^{s_j}.
\end{align*} 
The claim on the domain $ \DD^N_{X'}$ is also easy. 

{\it (2)} It is clear (and is already noticed). 

{\it (3)} Without loss of generality, $f_j(0)= 1$ for every $j\geq 1$ 
(indeed, $0<\prod_{j\geq 1}|f_j(0)|<\infty$). Observe first, that 
$f_j=:\sum_{k\geq 0}a_{j,k}\zeta_k\in\Mult(X_A)$ for every $j$: 
$$ 
\|f_j\|_\Mult = \left\|1+\sum_{k\geq 1}a_{j,k}\zeta_k\right\|_\Mult
\leq 1+\sum_{k\geq 1}|a_{j,k}|\cdot \|M_{\zeta_k}\| 
\leq 1+ 2^{\overline{s}} \|f_j-1\|_{\ell^1},
$$ 
where $\overline{s}=\max_{k\geq 1}(s_k)$. Since 
$\prod_j(1+\|f_j-1\|_{\ell^1})<\infty$, it follows that the product $f$, as 
well as all products $\prod_{j\geq n}f_j$, are also multipliers. Hence, 
for the cyclicity of $f$, it suffices to prove that every $f_j$ is cyclic. 

Simplifying the notation, let $f_j=F(\zeta)=1-l(\zeta)$ and 
$F_r=1-rl(\zeta)$, $0<r<1$. As in Proposition 7.2(2), we show that 
$\|1/F_r\|_X=O\big(\frac{1}{1-r}\big)$ for $r\to1$, which implies the 
result. Indeed, we have pointwise $1/F_r=\sum_{k\geq 0}r^kl(\zeta)^k$ and  
\begin{align*} 
\frac{1}{F_r} &= \sum_{k\geq 0}r^kl(\zeta)^k 
= \sum_{k\geq 0}r^k \sum_{|\alpha|=k}a^\alpha\zeta^\alpha
  \frac{|\alpha|!}{\alpha!},
\\ 
\left\|\frac{1}{F_r}\right\|^p 
&= \sum_{k\geq 0}r^{kp}\left\|\sum_{|\alpha|=k}a^\alpha\zeta^\alpha
  \frac{|\alpha|!}{\alpha!}\right\|^p
= \sum_{k\geq 0}r^{kp} \sum_{|\alpha|=k}|a^\alpha|^p
  \left(\frac{|\alpha|!}{\alpha!}\right)^p w_\alpha^p.
\end{align*} 
Maximizing $\log(w_\alpha(s))$ over all $\alpha\in\ZZ_+(\infty)$ with a 
given (finite) support $\sigma\subset\NN$ and $\sum\alpha_j=k$, we can 
restrict ourselves to the case $\#\sigma\leq k$, and then as in the proof 
of 7.2(2), see that $\max(w_\alpha(s))$ over all such $\alpha$ is  
\begin{multline*} 
\max\left\{ w_\alpha(s) : \supp(\alpha)\subset \sigma, \sum\alpha_j=k \right\}
= \prod_{j\in\sigma} \left(s_j\frac{\#\sigma+k}{\sum_{j\in\sigma}s_j}\right)^{s_j}
\\
\leq (2k)^{\sum_j s_j} \prod_{j\in \sigma}
  \left(\frac{s_j}{\sum_{j\in\sigma}s_j}\right)^{s_j}
\leq (2k)^{\sum_j s_j}\leq (2k)^{|s|} .
\end{multline*} 
Using $\sum_{|\alpha|=k}|a^\alpha|\big(\frac{|\alpha|!}{\alpha!}\big)\leq1$, 
we obtain $\|1/F_r\|^p\leq C\sum_{k\geq 0}r^{kp}k^{|s|p}$, and so 
$$ 
\|1/F_r\|^p_X = O\left(\frac{1}{(1-r)^{1+|s|p}}\right),
$$ 
where $1/p+|s|\leq 1$. This proves the claim. 
\end{proof} 
 
\section{Factorable functions in Dirichlet--Drury--Arveson spaces and 
algebras in \texorpdfstring{$\ell^p$}{lp}-unit ball}
 
Here we briefly consider an $\ell^p$ version of weighted Drury-Arveson space 
$\fD_0=\ell^2_{w_\alpha}(\ZZ_+^N)$ (see 6.10(4) above) and its infinite 
variables versions, in the form $\fD_s^p=\ell^p_{w_\alpha(s)}(\ZZ_+(\infty))$, 
$s\in \RR_+$, 
$$ 
\fD_s^p = \{ x=(x_\alpha)_{\alpha\in\ZZ_+(\infty)} : 
  (x_\alpha w_\alpha)\in\ell^p \}, 
w_\alpha = \left(\frac{\alpha!}{|\alpha|!}\right)^{1/p'} (|\alpha|+1)^s,
$$ 
where, as above, $\alpha!=\alpha_1!\alpha_2!\dots$, 
$|\alpha|=\alpha_1+\alpha_2+\dots$ ($ \alpha\in\ZZ_+(\infty)$), $1/p+1/p'=1$. 
The \emph{Drury--Arveson spaces are the Bohr transforms $B\fD^2_0$}
(defined in the unit balls $\BB^N$, see 8.1 below). Since the answers to 
the cyclicity question are different for the algebra- and 
non-algebra-cases of $\fD_s^p$, we state preliminary results separately 
(Lemmas 8.1 and 8.2). The following lemma identifies the domain of the 
definition for the holomorphic space $B\fD_s^p=(\fD_s^p)_A$. 
 
\begin{Lemma}
Let $X=\fD_s^p$ as defined above. Then, $X=\fD_s^p$ is an {\bf ISL}, and 
an elementary dilation $D_{e_j}$, $e_j=(\delta_{j,k})_{k\geq 1}$ has the 
norm $\|D_{e_j}\|\leq 2^s$.

For $sp'\leq 1$, the domain $\DD^\infty_{X'}$ is the open unit 
``$\ell^{p'}$-ball'' $\BB_{p'}^\infty$,
$$ 
\DD^\infty_{X'}=\BB_{p'}^\infty
=: \left\{ \lambda=(\lambda_j)_{j\geq 1}\in\CC^\infty : 
  \sum_{j\geq 1}|\lambda_j|^{p'}< 1 \right\},
$$ 
and for $sp'>1$ it is the closed ``$\ell^{p'}$-ball''
$$ 
\DD^\infty_{X'}=\overline{\BB}_{p'}^\infty 
=: \left\{ \lambda=(\lambda_j)_{j\geq 1}\in \CC^\infty : 
  \sum_{j\geq 1}|\lambda_j|^{p'}\leq 1 \right\}
$$ 
(for $p'=\infty$ and $s=0$, $X=\fD_0^1=\ell^1(\ZZ_+(\infty))$ and 
$\DD^\infty_{X'}=\overline{\DD}_\infty^\infty 
=: \{\lambda=(\lambda_j)_{j\geq 1}\in \CC^\infty:|\lambda_j|\leq 1\}$). 
The finite dimensional sections of $\DD^\infty_{X'}$ are the 
$\ell^{p'}$-balls $\BB_{p'}^N=\DD^\infty_{X'}\cap\CC^N$ (open or closed).
\end{Lemma}
 
\begin{proof}
Let $1<p<\infty$ (modifications for $p=1$ or $p=\infty$ are standard). 
Since $\|D_\alpha\| = \sup\big\{ \frac{w_{\alpha+\beta}}{w_\beta} 
  : \beta\in\ZZ_+(\infty)\big\}$, we have for $\alpha = e_j$: 
\begin{align*} 
\|D_{e_j}\| &= \|M_{\zeta_j}\| = \sup\left\{\frac{w_{e_j+\beta}}{w_\beta} 
  : \beta\in\ZZ_+(\infty)\right\} =
\\ 
&= \sup\left\{ \left( \frac{(\beta+e_j)!|\beta|!}{(|\beta|+1)!|\beta|!} 
  \right)^{1/p'} \left( \frac{|\beta|+2}{|\beta|+1} \right)^s : 
  \beta\in\ZZ_+(\infty) \right\} \leq 
\\
&\leq \sup\left\{ \left( \frac{|\beta|+2}{|\beta|+1}\right)^s : 
  \beta\in\ZZ_+(\infty) \right\} = 2^s .
\end{align*} 
The domain $\DD^N_{X'}$ consists of all $\lambda=(\lambda_j)_{j\geq 1}$ 
such that $(\lambda^\alpha)_\alpha\in\ell^{p'}_{1/w_\alpha(s)}(\ZZ_+(\infty))$. 
Summing up, 
\begin{align*} 
\sum_\alpha |\lambda^\alpha|^{p'} \frac{|\alpha|!}{\alpha!}(|\alpha|+1)^{-sp'}
&= \sum_{n\geq 0} (n+1)^{-sp'} \sum_{|\alpha|=n} |\lambda^\alpha|^{p'} 
  \frac{|\alpha|!}{\alpha!} =
\\ 
&= \sum_{n\geq 0} (n+1)^{-sp'} \left(\sum_{j\geq 1}|\lambda_j|^{p'}\right)^n,
\end{align*} 
which is finite if and only if either $\sum_{j\geq 1}|\lambda_j|^{p'}<1$ 
or $\sum_{j\geq 1}|\lambda_j|^{p'}\leq 1$ and $sp'>1$.
\end{proof}
 
In the next Lemma, we consider a slightly larger scale of spaces 
$\fD_{s,t}^p= \ell^p_{w_\alpha(s)}(\ZZ_+(\infty))$ introducing one more 
parameter as follows
$$ 
w_\alpha = w_\alpha(s,t) := \left( \frac{\alpha!}{|\alpha|!} \right)^t
  (|\alpha|+1)^s, \quad t,s\in\RR_+.
$$ 
Previously introduced weights are $w_\alpha=w_\alpha(s,1/p')$.

The key point for checking the condition of Lemma 7.1 for the 
$\fD_{s,t}^p$ spaces is the use of a binomial coefficient identity, namely 
a ``generalised Vandermond identity'' for binomial coefficients (coming from 
the famous ``Pascal's triangle/simplex''). This point was suggested to me by 
A.~Borichev (Aix-Marseille University) after reading a preliminary version 
of that paper. See also comments 8.3 below.
 
\begin{Lemma}
$\fD_{s,t}^p$ is a convolution algebra (a Banach algebra after an 
equivalent renorming) for
$$ 
s>1/p', \ t\geq 1/p' \text{ (for } p>1; \text{ and } s\geq 0, \ t\geq 0
\text{ for } p=1).
$$
\end{Lemma}
 
\begin{proof}
It suffices to verify the sufficient condition of Lemma 7.1. Let $p>1$ 
(modifications for $p=1$ will be obvious), $\gamma\in\ZZ_+(\infty)$, and 
$$ 
W_\gamma^{p'} 
= \sum_{\alpha+\beta=\gamma} \left(\frac{w_\gamma}{w_\alpha w_\beta}\right)^{p'},
$$ 
where $w_\alpha=w_\alpha(s,t)$. We employ the following identity for 
binomial (or multinomial) coefficients, sometimes called ``generalised 
Vandermond identity'' (see comments, references and a proof below): for 
$0\leq k\leq |\gamma|$,
$$ 
\sum_{\alpha+\beta=\gamma,|\alpha|=k} \frac{\gamma!}{\alpha!\beta!}
  = \binom{|\gamma|}{k},
$$ 
where $\binom{l}{k}$ stands for the binomial coefficients. Using this we 
obtain
$$ 
W_\gamma^{p'} = \sum_{k=0}^{|\gamma|} \sum_{\alpha+\beta=\gamma,|\alpha|=k}
  \left( \frac{\gamma!}{\alpha!\beta!}\right)^{tp'}c(k),
$$ 
where 
$$ 
c(k) = \left(\frac{k!(|\gamma|-k)!}{|\gamma|!}\right)^{tp'}
  \left(\frac{|\gamma|+1}{(k+1)(|\gamma|-k+1)}\right)^{sp'},
$$ 
and hence 
\begin{align*} 
W_\gamma^{p'} &\leq \sum_{k=0}^{|\gamma|} \binom{|\gamma|}{k}^{tp'} c(k)
  = \sum_{k=0}^{|\gamma|} 
   \left(\frac{|\gamma|+1}{(k+1)(|\gamma|-k+1)}\right)^{sp'} \leq 
\\ 
&\leq 2\sum_{k=0}^{|\gamma|/2} 
  \left(\frac{|\gamma|+1}{(k+1)(|\gamma|-k+1)}\right)^{sp'} 
\leq 2\sum_{k=0}^{|\gamma|/2} \left(\frac{2}{k+1}\right)^{sp'} <
\\
&< C^{p'} =: 2\sum_{k=0}^\infty \left(\frac{2}{k+1}\right)^{sp'} < \infty .
\end{align*} 
So, $\sup_\gamma W_\gamma<\infty$, and by Lemma 7.1, $\fD_{s,t}^p$ is an 
algebra.
\end{proof}
 
\begin{Comments}
The algebra property for the Hilbert space $\fD_{s,t}^2(\BB^N)$ in finite 
number of variables is proved for $t=1/p'=1/2$ and $s>1/2$ in 
\cite{OrF2006}, Theorem 4.1, and then --- with a very short proof --- in 
\cite{AHMcCR2022}, Lemma 14.2 (in both cases, using the reproducing kernel 
techniques and certain integration expressions for the norm coming from 
the Hilbert space structure). 

A binomial identity used in the proof of 8.2 (``generalised Vandermond's 
identity'') rose from the classical ``Pascal/Tartaglia's triangle/simplex'', 
and so goes back as early as to X-XIII centuries, see for this history, 
for example, \cite{Mer2003} or \cite{Pas-Wiki,Van-Wiki}. For completeness, 
we give a short proof.
\end{Comments}
 
\begin{proof}
Let $n_i\in\ZZ_+$, $m=\sum_{i=1}^Nn_i$. Using the Newton binomial, 
\begin{align*} 
\sum_{k=0}^m \binom{m}{k} x^k &= (1+x)^m = \prod_{i=0}^N(1+x)^{n_i}
  = \prod_{i=0}^N \left( \sum_{k_i=0}^{n_i} \binom{n_i}{k_i} x^{k_i} \right) = 
\\
&= \sum_{k=0}^mx^k \sum_{k_1+\dots+k_N=k,k_i\leq n_i} 
  \prod\binom{n_i}{k_i}.
\end{align*} 
Comparing the coefficients, we see that 
\begin{equation*} 
\binom{m}{k} = \sum_{k_1+\dots+k_N=k,k_i\leq n_i} \prod\binom{n_i}{k_i}.
\qedhere
\end{equation*} 
\end{proof}
 
Below, in Proposition 8.5 (an $\ell^{p'}$-unit ball analogue of 
Proposition 7.4), an extra property $f\in\ell^1_A$ of a function from 
$X=\fD_s^p$ is required. The following lemma, in a sense, justifies this 
requirement.
 
\begin{Lemma}
If a linear affine function $f$, 
$$ 
f(\zeta) = a+\sum_{k\geq 1}a_k\zeta_k
$$ 
is nonvanishing on every $\BB_{p'}^N$, $N= 1,2,\dots$, then 
$|a|\geq\big(\sum_{k\geq 1}|a_k|^p\big)^{1/p}$. 
\end{Lemma}
 
\begin{proof}
It is clear that the upper bound of the linear form 
$$ 
l(\zeta) = \sum_{k\geq 1}a_k\zeta_k
$$ 
over the ball $\BB_{p'}^N=:\{\zeta=(\zeta_j)_{j\geq 1}\in\CC^N : 
\sum_{j\geq 1}|\zeta_j|^{p'}<1\}$ is 
$R=:\big(\sum_{k=1}^N|a_k|^p\big)^{1/p}$, and the range $l(\BB_{p'}^N)$ is 
the disc $\{z\in\CC:|z|<R\}$. Since $|f(\zeta)|=|a+l(\zeta)|>0$ on 
$\BB_{p'}^N$, we have $|a|\geq\big(\sum_{k=1}^N|a_k|^p\big)^{1/p}$ for any 
$N\geq 1$, and so the claimed inequality.
\end{proof}
 
\begin{Proposition}
(1) Let $X=\fD_s^p$ be as above, and let $f\in X_A\cap\ell^1(\ZZ_+(\infty))_A$ 
be a factorable function in the sense that $f=\prod_{j\geq 1}f_j$, $f_j$ 
being linear affine functions, and
$$ 
0<\prod_{j\geq 1}\|f_j\|_{\ell^1}<\infty.
$$ 
Then if $s\leq 1/p'$ and $f(\zeta)\neq 0$ for all 
$\zeta\in\BB_{p'}^\infty$, then $f$ is $\sigma(X,X')$-cyclic in $X_A$. 

(2) Let $X=\fD_{s,t}^p$ be as above, and either $p>1$ and $s>1/p'$, 
$t\geq 1/p'$, or $p=1$ and $s\geq 0$, $t\geq 0$. Then, $X_A=B\fD_{s,t}^p$ 
is a function algebra on its maximal ideal space $\DD^\infty_{X'}$ (for 
$t=1/p'$, $\DD^\infty_{X'}=\overline{\BB}_{p'}^\infty$), and so a function 
$f\in X_A$ is cyclic if and only if $f(\zeta)\neq 0$, 
$\forall\zeta\in\DD^\infty_{X'}$(for $t= 1/p'$, 
$\forall\zeta\in \overline{\BB}_{p'}^\infty$). 
\end{Proposition}
 
\begin{proof}
{\it (1)}
Let $1<p<\infty$ (modifications for $p=1$ or $p=\infty$ are standard). 
Here, the computations similar to these of 7.2, 7.4 are applied. Without 
loss of generality, $f_j(0)=1$ for every $j\geq 1$. Also, all $f_j$ and 
$f$ are multipliers of $X_A$: $\|f_j\|_\Mult\leq 1+\|l_j\|_\Mult
\leq 1+\sum_{j\geq 1}|a_j|\cdot\|D_{e_j}\|\leq 2^s\|f_j\|_1$ (see Lemma
8.1). So, to prove point (1) it suffices to show that every $f_j$ is 
cyclic.

Simplifying the notation, let $f_j= F(\zeta)=1-l(\zeta)$ and 
$F_r=1-rl(\zeta)$, $0<r<1$. As in Proposition 7.4(3), we show that 
$\|1/F_r\|_X= O\big(\frac{1}{1-r}\big)$ for $r\to1$, which implies the 
result. Indeed, we have pointwise $1/F_r=\sum_{k\geq 0}r^kl(\zeta)^k$ and  
\begin{align*} 
\frac{1}{F_r} &= \sum_{k\geq 0}r^kl(\zeta)^k 
= \sum_{k\geq 0}r^k \sum_{|\alpha|=k}a^\alpha\zeta^\alpha\frac{|\alpha|!}{\alpha!},
\\
\left\|\frac{1}{F_r}\right\|^p 
&= \sum_{k\geq 0}r^{kp} \left\|\sum_{|\alpha|=k}a^\alpha\zeta^\alpha
  \frac{|\alpha|!}{\alpha!}\right\|^p 
= \sum_{k\geq 0}r^{kp} \sum_{|\alpha|=k} |a^\alpha|^p 
  \left(\frac{|\alpha|!}{\alpha!}\right)^p w_\alpha^p .
\end{align*} 
Since $w_\alpha^p = \big(\frac{\alpha!}{|\alpha|!}\big)^{p/p'}(|\alpha|+1)^{sp}$, 
we have 
\begin{align*} 
\left\|\frac{1}{F_r}\right\|^p &= \sum_{k\geq 0}r^{kp}k^{sp} 
  \sum_{|\alpha|=k} |a^\alpha|^p \frac{|\alpha|!}{\alpha!} =
\\
&= \sum_{k\geq 0}r^{kp}k^{sp}\left(\sum_{k\geq 1}|a_k|^p\right)^k
\leq \sum_{k\geq 0}r^{kp}k^{sp} = O\left(\frac{1}{(1-r)^{1+sp}}\right),
\end{align*} 
and so $\|1/F_r\|_X=O\big(\frac{1}{(1-r)^{s+1/p}}\big)$ as $r\to1$. This 
proves the claim. 

{\it (2)}
It is immediate from Lemmas 8.1 and 8.2. 
\end{proof} 
 
\begin{Rem}
One can show that $\overline{\BB}_{1/t}\subset\DD^\infty_{X'}
\subset\overline{\BB}_{p'}$ for any $t\geq 1/p'$.
\end{Rem}
 
\section{Brief notes on the history of the periodic Dilation 
Completeness Problem (DCP).}
 
Initially, the Dilation Completeness Problem consists in a search of 
functions $f\in L^p_{\loc}[0,\infty)$ such that
$$ 
E_f =: \Span_{L^p(0,1)} \left( f(nx) \big| (0,1):n=1,2,3,\dots\right)
  = L^p(0,1) .
$$ 
The famous Riemann Hypothesis (RH) on zeros of the $\zeta$-function is 
closely related to the DCP for $p=2$ and 
$f(x)=\frac{1}{x}-\big[\frac{1}{x}\big]$, $x>0$: (RH) is equivalent to the 
inclusion $1\in E_f$, and/or to the equality 
$\Span_{L^2(0,1)}(f(sx)|(0,1):s\geq 1)=L^2(0,1)$ (see 
\cite{Nym1950,Bae2003}, or Chapter 6 of \cite{Nik2019}). 

According to A.~Wintner \cite{Win1944} and A.~Beurling \cite{Beu1945}, the 
following partial case of the DCP ($2$-periodic DCP for $p=2$) is also 
related to some number theoretic questions (in Diophantine analysis): to 
determine \emph{odd $2$-periodic functions $f\in L^2_{odd}(-1,1)$} on 
$\RR$ such that $E_f=L^2(0,1)$. Since the functions 
$e_k=\frac{\sin(\pi kx)}{\sqrt{2}}$, $k=1,2,\dots$ form an orthonormal 
basis in $L^2_{\odd}(-1,1)$ and the dilations $f\mapsto f(nx)$ act as 
$e_k\mapsto e_{nk}$ ($n,k=1,2,\dots$), one can unitarily change the basis 
$(e_k)$ for $(z^k)$ ($k= 1,2,\dots$) on the Hardy space $H^2_0(\DD)$ on 
the unit disc $ \DD=\{z\in\CC:|z|<1\}$, 
$$ 
H^2_0(\DD) = \left\{ f=\sum_{k\geq 1}\hat{f}(k)z^k : 
  \|f\|^2=:\sum_{k\geq 1}|\hat{f}(k)|^2<\infty\right\},
$$ 
and get {\it the following equivalent form of the periodic DCP: 
to describe $f\in H^2_0(\DD)$ such that the dilations $T_nf$ 
$$ 
T_nf=f(z^n)
$$ 
generate the whole space,}
$$ 
E_f =: \Span_{H^2_0} (T_nf:n=1,2,\dots) = H^2_0;
$$ 
such a function $f$ is called \emph{``dilation cyclic''}. As we saw above, 
the problems of cyclic (polynomially cyclic, or weakly invertible) 
functions in one or several complex variables are often partial cases of 
the DCP in sequence spaces on $\NN$. These partial cases are treated in a 
lot of papers and books; quoting just a few, 
\cite{Rud1969,Rud1980,Nik1986,Zhu2005,EKMR2014} Ch.~9. The most facts and 
references known on the periodic DCP are gathered in 
\cite{HLS1997,HLS1999,Mit2017,Nik2012,Nik2018,Nik2018b,Nik2019,DGu2020,DGu2021} 
(and other Dan--Guo arXiv papers), \cite{DGN2021}, as well as in 
references of Section 6.10 above. B.~Mityagin pointed out (see for example 
\cite{Mit2017}) to an applied interest of the dilation cyclicity of 
binomials $f=z(\lambda-z)^N$ ($\lambda\in\CC$, $N= 1,2,\dots$) related to 
the Cohn--Lehmer--Schur algorithm (localization of zeros of polynomials). 
This special choice of $ f$ was already partially treated in 
\cite{Nik2019} and \cite{Nik2017}.

\end{document}